\documentclass[leqno, 11pt]{amsart}
\usepackage{color,mathrsfs}
\usepackage[utf8]{inputenc}

\usepackage[colorlinks=true, pdfstartview=FitV, linkcolor=blue, citecolor=blue, urlcolor=blue,pagebackref=false]{hyperref}
\usepackage{calc,amsfonts,amsthm,amscd,epsfig,psfrag,amsmath,amssymb,enumerate,graphicx}
\usepackage{tikz}
\usepackage{pgfplots}

\newlength\fullwidth
\numberwithin{equation}{section}

\DeclareMathSymbol{\leqslant}{\mathalpha}{AMSa}{"36} 
\DeclareMathSymbol{\geqslant}{\mathalpha}{AMSa}{"3E} 
\DeclareMathSymbol{\eset}{\mathalpha}{AMSb}{"3F}     
\renewcommand{\leq}{\;\leqslant\;}                   
\renewcommand{\geq}{\;\geqslant\;}                   
\DeclareMathOperator{\capacity}{\mathrm{cap}}       
\DeclareMathOperator{\Gap}{\mathrm{gap}}       
 \def\1{\ifmmode {1\hskip -3pt
    \rm{I}} \else {\hbox {$1\hskip -3pt \rm{I}$}}\fi}

\newcommand{\var}{\operatorname{Var}}

\newcommand{\Z}{\mathbb{Z}}

\renewcommand{\d}{\delta}

\newcommand{\gb}{\beta}

\newcommand{\z}{\zeta}

\newcommand{\gap}{{\rm gap}}

\newtheorem{theorem}{Theorem}[section]

\newtheorem{lemma}[theorem]{Lemma}
\newtheorem{proposition}[theorem]{Proposition}

\newtheorem{remark}[theorem]{Remark}

\newcommand{\bP}{{\bf P}}

\newcommand{\cA}{\ensuremath{\mathcal A}}

\newcommand{\cE}{\ensuremath{\mathcal E}}

\newcommand{\cL}{\ensuremath{\mathcal L}}

\newcommand{\cS}{\ensuremath{\mathcal S}}

\newcommand{\bbN}{{\ensuremath{\mathbb N}} }

\newcommand{\bbR}{{\ensuremath{\mathbb R}} }

\newcommand{\bbZ}{{\ensuremath{\mathbb Z}} }

\newcommand{\gep}{\varepsilon}

\newcommand{\Tr}{T_{\rm rel}}
\newcommand{\tf}{\textsc{f}}

\newcommand{\ind}{{\bf 1}}

\newcommand{\dd}{\mathrm{d}}

\newcommand{\gl}{\lambda}

\renewcommand{\ge}{\geq}
\renewcommand{\le}{\leq}

  \newcounter{constant}

\usepackage{array}
\newcolumntype{e}{>{\displaystyle}r @{\,} >{\displaystyle}c @{\,} >{\displaystyle}l}

\begin{document}
\begin{abstract}
  In this paper we investigate the dynamical behavior of a polymer interface, in interaction with a distant attractive substrate. The interface is modeled by the graph of a nearest neighbor path with non-negative integer coordinates, and the equilibrium measure associates to each path $\eta$ a probability proportional to $\gl^{H(\eta)}$ where $\gl\in \bbR_+$ and $H(\eta)$ is the number of contacts between $\eta$ and the substrate.  The dynamics is the natural ``spin flip'' dynamics associated to this equilibrium measure.  We let the distance to the substrate at both polymer ends be equal to $aN$ where $a\in (0,1/2)$ is a fixed parameter, and $N$ is the length the system.  With this setup, we show that the dynamical behavior of the system crucially depends on $\gl$: when $\gl\le \frac{2}{1-2a}$ we show that the system only needs a time which is polynomial in $N$ to reach its equilibrium state, whereas $\gl> \frac{2}{1-2a}$ the mixing time is exponential in $N$ and the system relaxes in an exponential manner 
which is typical of metastability.
\end{abstract}

\title[Metastable wetting]{A mathematical perspective on metastable wetting}

\author[H. Lacoin]{Hubert Lacoin}
\address{H. Lacoin,
CEREMADE - UMR CNRS 7534 - Universit\'e Paris Dauphine,
Place du Mar\'echal de Lattre de Tassigny, 75775 CEDEX-16 Paris, France. \newline
e--mail: {\tt lacoin@ceremade.dauphine.fr}}

\author[A. Teixeira]{Augusto Teixeira}
\address{A. Teixeira, Instituto Nacional de Matem\'atica Pura e Aplicada,
Estrada Dona Castorina 110,  Jardim Botanico, Cep 22460-320,
Rio de Janeiro, Brazil \newline e--mail: {\tt augusto@impa.br}}


\maketitle

\section{Introduction}

The aim of this paper is to study the dynamics of a model for an interface interacting with a substrate. This study was partially inspired by a recent work in theoretical physics \cite{CDH11} which proposed a model to account for metastable transition for wetting of droplets on a grooved surface. The origin of the metastable behavior is the following: consider a droplet that lies on the top of a surface cavity (see Figure \ref{fig:wetting}). If the substrate is energetically favorable, then the lowest energy state is the one where the droplet wets the bottom of the cavity. However, to reach this state, the droplet primarily has to increase its surface tension, and thus to overcome an energy barrier. For this reason, the droplet will remain above the cavity for some time, until some perturbation helps it perform the transition.

\medskip

In \cite{CDH11} this situation was reduced to a $1+1$ toy model in order to make some qualitative description of the relaxation to equilibrium of the droplet. The object of this paper is to bring this description on rigorous ground. For technical convenience, we study a model that slightly differ from the one in \cite{CDH11} in the sense that it is based on the simple random walk pinning model instead of the so-called Solid-On-Solid model.

\medskip

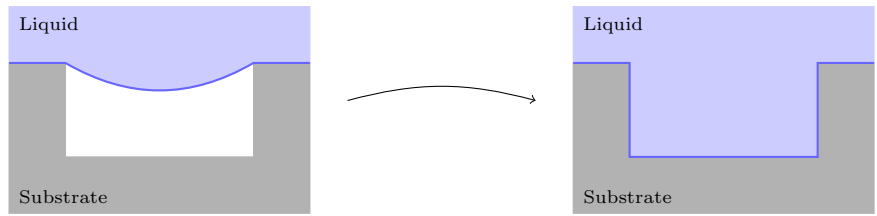
\begin{figure}[h]
\centering
  \begin{tikzpicture}[scale=.25,font=\tiny]
    \draw[fill, color=gray!60!white] (0,5) -- (3,5) -- (3,0) -- (13,0) -- (13,5) -- (16,5) -- (16,-3) -- (0,-3) -- (0,5);
    \draw[fill, color=gray!60!white] (30,5) -- (33,5) -- (33,0) -- (43,0) -- (43,5) -- (46,5) -- (46,-3) -- (30,-3) -- (30,5);
    \draw[fill, color=blue!20!white] (0,5) -- (3,5) to [out=330,in=210] (13,5) -- (16,5) -- (16,8) -- (0,8) -- (0,5);
    \draw[thick, color=blue!60!white] (0,5) -- (3,5) to [out=330,in=210] (13,5) -- (16,5);
    \draw[fill, color=blue!20!white] (30,5) -- (33,5) -- (33,0) -- (43,0) --  (43,5) -- (46,5) -- (46,8) -- (30,8) -- (30,5);
    \draw[thick, color=blue!60!white] (30,5) -- (33,5) -- (33,0) -- (43,0) -- (43,5) -- (46,5);
    \draw (18,3) edge[out=15,in=165,->] (28,3);
    \node[below right] at (0,8) {Liquid};
    \node[below right] at (30,8) {Liquid};
    \node[above right] at (0,-3) {Substrate};
    \node[above right] at (30,-3) {Substrate};
  \end{tikzpicture}
  \caption{In this work, we model the transition corresponding to the moment when the interface of the liquid wets the bottom of the substrate cavity.}
  \label{fig:wetting}
\end{figure}

The random walk pinning model has been introduced several decades ago (see the seminal paper \cite{Fischer84})
and has been the object of a large number of studies, both in its homogeneous and disordered
versions (see \cite{Gia07} or \cite{Giaflour} for recent reviews).

\medskip

The prototypical random walk pinning model is defined as follows:
Given $N\in 2\bbN$, we set
\begin{equation}
 \mathcal S_N:= \Big\{ \eta = (\eta_x)_{x\in [0,N]} \ | \ \eta_0=\eta_N=0 \text{ and $\forall x \in [0,N]$, } \eta_x \in \Z_+ \text{ and } |\eta_{x+1}-\eta_{x}|=1\Big\}.
\end{equation}
The graph of $\eta$ models an interface or polymer that stretches in the horizontal direction.
The constraint $\eta_x \in \Z_+$ materializes the fact our polymer cannot visit the  half space $[0,N]\times \bbZ_{-}$ which is occupied by a solid substrate or wall.
Given $\eta \in \mathcal{S}_N$, we define $H(\eta)$ to be the number of contact points of the graph of $\eta$ with the wall (we refer to Figure \ref{fig:scaling} for a graphical representation of the polymer).
\begin{equation}
 H(\eta) = \sum_{0 \leq x \leq N} \ind_{\{\eta_x = 0\}}
\end{equation}
and for $\lambda \in \mathbb{R}$, the corresponding Gibbs measure is given by
\begin{equation}
  \label{e:pi_N_Lambda}
  \pi^\lambda_N(\eta) = \frac 1{Z^\lambda_N} \lambda^{H(\eta)},
\end{equation}
where $\lambda$ is what we call the pinning parameter. It is equal to $\exp(-E/k_B T)$ where $E$ is the energy of interaction of the monomers with the wall, $T$ is the temperature, and
$k_B$ is the Boltzman constant. Hence $ \lambda^{H(\eta)}$ correspond to the Boltzman weight associated to a trajectory $\eta$.
In \eqref{e:pi_N_Lambda}, $Z^\lambda_N$ is the normalizing constant which makes $\pi_N$ a probability measure, it is called the \emph{partition function} of the system,
\begin{equation}
 Z^\lambda_N:=\sum_{\eta \in \mathcal S_N} \lambda^{H(\eta)}.
\end{equation}

\medskip

The model was introduced to analyze the wetting transition for polymers interacting with an attractive substrate.
This transition can be observed through the study of the \emph{free-energy}, defined as
\begin{equation}
 \tf(\lambda):=\lim_{N\to \infty} \left(\frac{1}{N} \log Z^\lambda_N\right)-\log 2.
\end{equation}
We have a simple explicit expression for $F(\gl)$ (see \cite{MR2504175}, (7.14), (7.24) and (7.46)):
\begin{equation}
\tf(\gl)=\log \left(\frac{\gl}{ 2 \sqrt{\gl-1}}\right)\ind_{\gl>2}.
\end{equation}
It can be shown that for large $N$ under $\pi_N$, $\eta$ has asymptotically a positive contact fraction $H(\eta)/N$
if  $\tf(\lambda)>0$ (that is $\gl>2$) and that the contact fraction vanishes for $\gl \leq 2$ (see \cite[Chapter 2]{Gia07}).
The phase transition is said to be of order two as $\tf$ and its derivative are continuous at $\gl=2$.

\medskip

The dynamical version of this model has been investigated only more recently. The dynamics is a Markov chain on $\cS_N$ for which $\pi_N$ is the
invariant measure and whose transition are given by updates of local coordinate (see Section \ref{sec:dyna} for a formal description). The dynamics are usually studied to understand how a system relaxes to equilibrium.
In \cite{CMT08}, the authors proved that the mixing-time of the polymer dynamics on $\cS_N$ is of order $N^2$ (up to logarithmic correction) for every $\gl$.
The scaling limit of the polymers profile under diffusing scaling was investigated in \cite{Lac13}.

\medskip

In the present work, we study the effect of elevated boundary condition on the dynamics.
For $a\in (0,1/2)$ we define
\begin{equation}\label{e:SNA}
 \mathcal S^a_N:=\big\{ \eta = (\eta_x)_{x\in [0,N]} \ | \ \eta_0=\eta_N= \langle aN \rangle, \forall x\in [0,N], \ \eta_x\ge 0 \text{ and } |\eta_{x+1}-\eta_{x}|=1\big \},
\end{equation}
 where $\langle s \rangle$ denotes the smallest even integer larger or equal to $s$.

We define the Gibbs measure for the polymer with elevated boundary condition as follows
\begin{equation}
  \label{e:pi_lambda}
  \pi^{\lambda,a}_N (\eta):= \frac 1{Z^{\lambda,a}_N} \lambda^{H(\eta)},
\end{equation}
where the partition function $Z^{\lambda,a}_N$ is given by
\begin{equation}
  Z^{\lambda,a}_N:=\sum_{\eta \in \mathcal S^a_N} \lambda^{H(\eta)}.
\end{equation}

\medskip

If $\gl>2$, from the results on the model with standard boundary condition, the walk  is locally attracted to the wall.
However, because of the boundary condition, reaching the energetically favorable wall has an entropic cost, and there is a non-trivial competition between energy and entropy.

Before going into the details of the dynamics we study in detail the equilibrium distribution under \eqref{e:pi_lambda}, see Section~\ref{ss:statics}.
In particular we must identify the local equilibrium states of the polymer, which can be informally described as follows.

\medskip

When $\gl$ is sufficiently large (how large exactly is made explicit in Section \ref{sec:metastates}), as a result of this competition the polymer has two possible local equilibrium states (or phases) that are separated by a bottleneck. Let us give a more precise description of both. For each $a \geq 0$, there exists a critical pinning force $\gl_c(a)$ (strictly increasing in $a$) such that
\begin{description}
 \item [\quad Free] ($\gl < \gl_c$) In this phase, the height of the polymer has fluctuation order $N^{1/2}$ around the attaching height $aN$ and it stays unaware of the attractive wall at zero.
 \item [\quad Pinned] ($\gl \geq \gl_c$) In that phase, the polymer drives from $aN$ to zero with optimal slope $-d_\lambda$, see \eqref{e:dlamb}, then presents a pinned region that has macroscopic length (i.e.\ of order $N$) where it stays within distance of order $\log N$ from the wall and
 finally, it returns to height $aN$ with slope $d_\lambda$.
 Of course this can occur only if $a<d_{\gl}$ (see Figure \ref{fig:scaling}).
\end{description}
See also Figure~\ref{fig:scaling} for an illustration of these two phases.

Which of these local equilibrium state is the more favorable depends on the values of $\gl$ and $a$, see Figure~\ref{fig:phase_diagram}.
The above statements are made precise in Theorem~\ref{th:scaling}, where we provide a scaling limit of the polymer as $N$ goes to infinity in each of the above phases.

\medskip

\begin{figure}[hlt]
\centering
  \begin{tikzpicture}[scale=.05,font=\tiny]
    \draw[<->] (-2,30) -- (-2,0) -- (150,0) -- (150,30);
    \draw[dotted, color=black] (-2,21) -- (50,0);
    \draw[dotted, color=black] (150,21) -- (98,0);
    \draw[color=blue] (-2,21) -- (-1,20) -- (0,19) -- (1,18) -- (2,17) -- (3,16) -- (4,15) -- (5,14) -- (6,13) -- (7,14) -- (8,13) -- (9,12) -- (10,13) -- (11,12) -- (12,11) -- (13,10) -- (14,9) -- (15,10) -- (16,9) -- (17,10) -- (18,9) -- (19,8) -- (20,7) -- (21,8) -- (22,9) -- (23,10) -- (24,9) -- (25,10) -- (26,9) -- (27,8) -- (28,9) -- (29,8) -- (30,7) -- (31,6) -- (32,5) -- (33,6) -- (34,7) -- (35,6) -- (36,5) -- (37,4) -- (38,3) -- (39,2) -- (40,3) -- (41,2) -- (42,1) -- (43,2) -- (44,1) -- (45,2) -- (46,1) -- (47,0) -- (48,1) -- (49,0) -- (50,1) -- (51,0) -- (52,1) -- (53,0) -- (54,1) -- (55,0) -- (56,1) -- (57,0) -- (58,1) -- (59,0) -- (60,1) -- (61,0) -- (62,1) -- (63,0) -- (64,1) -- (65,2) -- (66,1) -- (67,0) -- (68,1) -- (69,0) -- (70,1) -- (71,2) -- (72,1) -- (73,0) -- (74,1) -- (75,0) -- (76,1) -- (77,0) -- (78,1) -- (79,2) -- (80,1) -- (81,0) -- (82,1) -- (83,2) -- (84,3) -- (85,2) -- (86,3) -- (87,2) -- (88,1) -- (89,0) -- (90,1) -- (91,0) -- (92,1) -- (93,2) -- (94,3) -- (95,2) -- (96,3) -- (
97,2) -- (98,3) -- (99,4) -- (100,5) -- (101,4) -- (102,3) -- (103,2) -- (104,3) -- (105,4) -- (106,5) -- (107,6) -- (108,7) -- (109,6) -- (110,5) -- (111,6) -- (112,5) -- (113,4) -- (114,5) -- (115,6) -- (116,7) -- (117,8) -- (118,7) -- (119,8) -- (120,9) -- (121,10) -- (122,9) -- (123,10) -- (124,11) -- (125,12) -- (126,11) -- (127,12) -- (128,13) -- (129,14) -- (130,15) -- (131,16) -- (132,17) -- (133,18) -- (134,19) -- (135,20) -- (136,19) -- (137,18) -- (138,19) -- (139,18) -- (140,17) -- (141,16) -- (142,17) -- (143,18) -- (144,19) -- (145,20) -- (146,21) -- (147,20) -- (148,21) -- (149,20) -- (150,21);
    \draw[fill] (-2,21) circle [radius=0.2];
    \draw[fill] (150,21) circle [radius=0.2];
    \node[left] at (-2,21) {$\langle aN \rangle$};
    \node[below] at (-2,0) {$0$};
    \node[below] at (150,0) {$N$};
    \node[right] at (150,21) {$\langle aN \rangle$};
    \draw[<->] (-2,69) -- (-2,39) -- (150,39) -- (150,69);
    \node[left] at (-2,61) {$\langle aN \rangle$};
    \node[right] at (150,61) {$\langle aN \rangle$};
    \draw[dotted, color=black] (-2,60) -- (150,60);
    \draw[color=blue] (-2,60) -- (-1,61) -- (0,60) -- (1,59) -- (2,60) -- (3,59) -- (4,58) -- (5,57) -- (6,58) -- (--7,59) -- (8,58) -- (9,59) -- (10,58) -- (11,57) -- (12,58) -- (13,59) -- (14,60) -- (15,61) -- (--16,62) -- (17,61) -- (18,62) -- (19,61) -- (20,60) -- (21,61) -- (22,60) -- (23,61) -- (24,60) -- (25,59) -- (26,60) -- (27,59) -- (28,60) -- (29,61) -- (30,62) -- (31,63) -- (32,64) -- (33,65) -- (34,64) -- (35,65) -- (36,66) -- (37,67) -- (38,66) -- (39,65) -- (40,66) -- (41,67) -- (42,68) -- (43,67) -- (44,66) -- (45,65) -- (46,66) -- (47,65) -- (48,64) -- (49,65) -- (50,66) -- (51,67) -- (52,66) -- (53,67) -- (54,68) -- (55,69) -- (56,68) -- (57,67) -- (58,66) -- (59,65) -- (60,66) -- (61,67) -- (62,68) -- (63,67) -- (64,66) -- (65,65) -- (66,66) -- (67,67) -- (68,66) -- (69,65) -- (70,66) -- (71,67) -- (72,66) -- (73,67) -- (74,66) -- (75,65) -- (76,66) -- (77,65) -- (78,66) -- (79,65) -- (80,66) -- (81,65) -- (82,64) -- (83,63) -- (84,62) -- (85,61) -- (86,62) -- (87,61) -- (88,60) -- (89,61)
 -- (90,60) -- (91,61) -- (92,62) -- (93,61) -- (94,60) -- (95,61) -- (96,60) -- (97,61) -- (98,62) -- (99,63) -- (100,62) -- (101,61) -- (102,62) -- (103,63) -- (104,64) -- (105,63) -- (106,62) -- (107,63) -- (108,64) -- (109,63) -- (110,64) -- (111,63) -- (112,64) -- (113,63) -- (114,62) -- (115,63) -- (116,62) -- (117,63) -- (118,62) -- (119,63) -- (120,64) -- (121,65) -- (122,66) -- (123,65) -- (124,64) -- (125,63) -- (126,64) -- (127,65) -- (128,66) -- (129,65) -- (130,66) -- (131,65) -- (132,66) -- (133,67) -- (134,66) -- (135,65) -- (136,64) -- (137,63) -- (138,62) -- (139,61) -- (140,62) -- (141,63) -- (142,64) -- (143,63) -- (144,62) -- (145,63) -- (146,62) -- (147,61) -- (148,62) -- (149,61) -- (150,60);
  \end{tikzpicture}
 \caption{\label{fig:scaling} Typical behavior for $\eta$ at equilibrium when $\gl < \gl_c(a)$ (free phase at the top) and $\gl \geq \gl_c(a)$ (pinned phase at the bottom). The dotted line illustrates $f_{a,\gl}$, which is the scaling limit when $N\to \infty$.}
\end{figure}
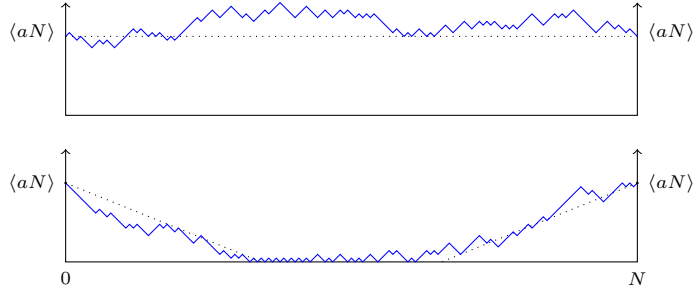

The main objective of this paper is to describe the behavior of this system under the heat bath dynamics for this polymer model. The precise definition of the generator is given in \eqref{e:gen}, although a quick look at Figure~\ref{fig:jumprates} already gives a good idea of the definition of the jump rates.

This system presents three distinct behaviors under the heat bath dynamics as $N$ grows, depending on the specific choice of $a$ and $\gl$. More precisely, there exist three regions in the phase diagram (see Figure~\ref{fig:phase_diagram}) that we informally describe as follows:
\begin{enumerate}[(a)]
\item Free phase - there are no bottlenecks for the dynamics and the polymer relaxes to equilibrium in polynomial time (this is of course also the case when
$\gl\le 2$ for which the polymer is not even locally attracted by the wall due to entropic repulsion).
\item Free phase (double well) - as above, the polymer does not attach to the substrate when at equilibrium, however, if one starts the system at a pinned configuration, it will take a long time (exponential in $N$) for it to reach the free phase.
\item Pinned phase (double well) - in this phase, the system stays pinned at equilibrium, but if one sets the initial condition at the free phase, the polymer takes an exponential time to attain the attractive wall.
\end{enumerate}
The precise formulation of these statements can be found in Theorem~\ref{th: mainresult1} below. The regions (b) and (c) present what we call metastable behavior (should one start the system from the local equilibrium phase). In Theorem~\ref{th: mainresult2} we show that in this case, the time to observe the transition to equilibrium converges to an exponential random variable when properly rescaled.

\medskip

\textbf{Acknowledgments -} We would like to thank Claudio Landim for helpful discussions on metastability and Makiko Sasada for indicating references \cite{BFO09} and \cite{FO10}. This work was initiated during the stay of H.L in IMPA researcher, he acknowledges kind hospitality and the support of CNPq.
A.T. is also grateful to the Brazilian-French Network in Mathematics for the opportunity to visit Paris during the elaboration of this work and for the financial support from CNPq, grants 306348/2012-8 and 478577/2012-5.

\section{Model and results}

\subsection{Statics for the system with elevated boundary conditions}
\label{ss:statics}

In order to study the behavior of the system at equilibrium it is natural to define its free energy (whose existence is ascertained by Proposition \ref{th:freeen} below) as
\begin{equation}\label{e:fre}
  \tf(\lambda,a):=\lim_{N\to \infty} \left(\frac{1}{N} \log Z^{\lambda,a}_N\right)-\log 2.
\end{equation}


In order to derive the expression   $\tf(\lambda,a)$, we must evaluate the cost for the polymer to drift-down with a given slope $d$ until it meets the wall, and then try to optimize this scheme by taking the maximum over $d$. This is done in the next proposition, which gives the following expression for the free-energy.


\begin{proposition}\label{th:freeen}
The free energy of the system with elevated boundary condition defined by Equation \eqref{e:fre} exists. It is the solution of the following optimization problem,
\begin{equation}\label{e:varp}
   \tf(\lambda,a):=\max \Big \{0,\max_{d\in [2a,1]}\Big(\tf(\lambda) \big (1-\tfrac{2a}{d}\big ) -\tfrac{2a}{d}q(d)\Big)\Big\},
\end{equation}
where $q$ is defined on $[0,1]$ as follows
\begin{equation}
\begin{split}
 q(d):= & -\lim_{N\to \infty} \frac{1}{N}\log \big | \big \{ \text{ simple paths of length $N$ linking $0$ to $\langle dN \rangle$ } \big \} \big|+\log 2\\
 = & \frac{1}{2} \big[ (1+d)\log(1+d)+(1-d)\log(1-d) \big] \ge 0.
\end{split}
\end{equation}
If   $\tf(\lambda,a)$ is positive (and thus $\gl>2$),  then maximum $\max_{d\in [2a,1]} \Big(\tf(\lambda) \big (1-\tfrac{2a}{d}\big ) -\tfrac{2a}{d}q(d)\Big) $ is attained
when
\begin{equation}\label{e:dlamb}
d=d_\lambda:=\sqrt{1-\exp(-2\tf(\lambda))}=1-\frac{2}{\gl}.
\end{equation}
Hence we also have
\begin{equation}\label{e:express}
   \tf(\lambda,a):=\Big(\tf(\lambda)\big(1-\tfrac{2a}{d_\lambda}\big)-\tfrac{2a}{d_\lambda}q(d_\lambda)\Big)_+
   =\left(\tf(\gl)- a \log\left(\frac{1+d_\gl}{1-d_\gl}\right)\right)_+.
  \end{equation}
The function $\tf(\lambda,a)$ is analytic in $a$ and $\lambda$ except on the critical curve $\gl=\gl_c(a)$, determined by the unique solution $\gl_c(a)$ of the equation
$$\tf(\gl)= a \log\left(\frac{1+d_\gl}{1-d_\gl}\right)= a \log (\gl-1)$$
The right derivative of $\tf(\gl,a)$ at $\gl=\gl_c(a)$, is positive, and thus the phase transition in $\gl$ is of first order.

\end{proposition}

The above proposition is proved in Section~\ref{sec:prel}.

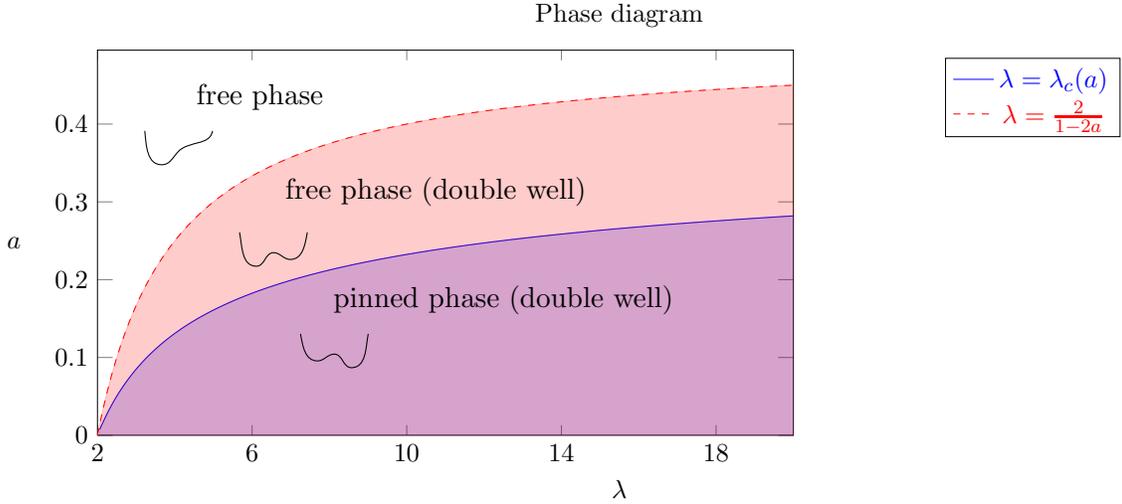
\begin{figure}[h]
  \centering
  \begin{tikzpicture}[scale=.9]
    \begin{axis}[
      xmin=2, xscale=1.5, xmax=20, domain=2:20, samples=300,
      ymin=0, ylabel style={rotate=-90}, 
      xlabel=$\lambda$, ylabel=$a$, title={Phase diagram},
      xtick={2,6,10,14,18}
    ]
    \addplot[color=blue]
            {ln(x/(2*sqrt(x-1))) / ln(x-1)};
    \addlegendentry[blue,thick] {$\lambda = \lambda_c(a)$}
    \addplot[dashed,color=red]
            {0.5 - 1/x};
    \addlegendentry[red,thick] {$\lambda = \frac{2}{1 - 2a}$}
    \addplot[color=blue,fill,opacity=.2]
            {ln(x/(2*sqrt(x-1))) / ln(x-1)} \closedcycle;
    \addplot[color=red,fill,opacity=.2]
            {0.5 - 1/x} \closedcycle;
    \end{axis}
    \node at (2.4,5) {$\text{free phase}$};
    \draw (0.7,4.5) to [out=275,in=180] (.95,4) to [out=0,in=225] (1.2,4.2) to [out=45,in=260] (1.7,4.5);
    \node at (5,3.6) {$\text{free phase (double well)}$};
    \draw (2.1,3.5-0.5) to [out=280,in=180] (2.35,3-0.5) to [out=0,in=180] (2.6,3.2-0.5) to [out=0,in=180] (2.85,3.1-0.5) to [out=0,in=260] (3.1,3.5-0.5);
    \node at (6,2) {$\text{pinned phase (double well)}$};
    \draw (3,3.5-2) to [out=280,in=180] (3.25,3.1-2) to [out=0,in=180] (3.5,3.2-2) to [out=0,in=180] (3.75,3-2) to [out=0,in=260] (4,3.5-2);
  \end{tikzpicture}
  \caption{The phase diagram for the polymer close to an attractive wall. The red line $\lambda = \frac{2}{1-2a}$ separates the fast mixing free phase from the free phase with double well (where metastability is observed). The pinned phase is determined by $\lambda > \lambda_c(a)$, which is equivalent to the condition $a < \frac{\log(\lambda/2\sqrt{\lambda-1})}{\log(\lambda-1)}.$}
  \label{fig:phase_diagram}
\end{figure}


From these results, we can also deduce the typical behavior of $S$ under $\pi^{\lambda,a}_N$: it says that when $\gl>\gl_c$ the polymer typically drift towards the wall with a slope $d_{\gl}$ on both sides and presents a pinned region in the middle which is of length $N(1-2a/d_{\gl})$, see Figure~\ref{fig:scaling} (bottom), whereas when $\gl<\gl_c$ the polymer typically lies in the free phase.  For the case $\gl=\gl_c$, estimates on the exponential scale are not sufficient to decide in which phase the polymer lies.  However, our proofs contain finer estimates and allows us to establish that when $\gl=\gl_c$ the polymer is typically pinned (see Proposition \ref{critbeha}).  From the proof of the above proposition, one can derive a scaling limit result for the polymer at equilibrium.  When $\gl> \frac{2}{1-2a}$, we set

\begin{equation}
f_{a,\gl}(x):=\max(a-d_\gl x,0,a+d_{\gl}(x-1)).
\end{equation}

\begin{theorem}\label{th:scaling}
When $\gl\ge \gl_c(a)$ we have for all $\gep>0$
\begin{equation}\label{scal1}
\lim_{N\to \infty} \pi^{\gl,a}_N\left(\max_{x\in[0,1]} \left|\frac{1}{N}\eta(Nx)-f_{a,\gl}(x) \right|\ge \gep \right)=0,
\end{equation}
when $\gl< \gl_c(a)$
\begin{equation}\label{scal2}
\lim_{N\to \infty} \pi^{\gl,a}_N\left(\max_{x\in[0,1]} \left|\frac{1}{N}\eta(Nx)-a \right|\ge \gep \right)=0,
\end{equation}
\end{theorem}

The proof of the above result will be provided in Subsection~\ref{ss:scaling_limit}.


In fact with only a minor additional effort one could in principle prove a large deviation principle for the rescaled path $\frac{1}{N}\eta(Nx)$, when $\gl \neq \gl_c$. However, this is not in the scope of this paper.
Let us mention \cite{BFO09} where an LDP was proved for a continuous wetting model with elevated boundary condition  (see also \cite{FO10} which focuses on the case $\gl=\gl_c$).

\subsection{Dynamics}\label{sec:dyna}

Let us now introduce the generator $\mathcal{L}^\lambda$ of our dynamics, which corresponds to a heat bath of our polymer. For this, given a polymer $\eta \in \mathcal{S}_N$ and $1 \leq x \leq N-1$, we define the polymer with corner flipped at $x$ by
\begin{equation}
\begin{cases}
\eta^{x}_x&=\eta_{x+1}+\eta_{x-1}-\eta_x,\\
\eta^x_y&=\eta_y, \quad \text{ for } x \neq y.
\end{cases}
\end{equation}

The operation $\eta\to \eta^x$ transforms a local maximum at $x$ into a local minimum (respectively local minimum into a local maximum).

Let $\mathcal{S}$ be a space of polymers with length $N$ (with either zero or elevated boundary conditions). The generator $\mathcal{L}^\lambda$ acts on $f:\mathcal{S} \to \mathbb{R}$ as follows
\begin{equation}
 \label{e:gen}
 (\mathcal{L}^\lambda_\mathcal{S} )f(\eta) := \sum_{\eta'\in \mathcal S} r^\gl(\eta,\eta') \big(f(\eta') - f(\eta)\big)
 = \sum_{x=1}^{N-1} r^\gl(\eta,\eta^x) \big(f(\eta^x) - f(\eta)\big),
\end{equation}
where the rates $r^\gl$ are given by (see also Figure~\ref{fig:jumprates} for a graphical representation of the jump rates)
\begin{equation}\begin{split}\label{e:defrates}
 r^\lambda(\eta,\eta^x) &=
 \begin{cases}
  \tfrac 12 \quad & \text{if $\eta_{x}$ and $\eta_x^x > 0$} ,\\
    \tfrac \gl{\lambda+1} &\text{if $\eta_{x\pm 1}  = 1$ and $\eta_x=2$,}\\
  \tfrac 1{\lambda+1} &\text{if $\eta_x = 0$,}\\
  0 & \text{if } \eta_{x\pm 1}=0 \end{cases}\\
  r^\gl(\eta,\eta')&=0 \quad  \text{ if } \eta'\notin \{ \eta^x \ | \ x\in \{1,\dots, N-1\}\}.
  \end{split}
\end{equation}


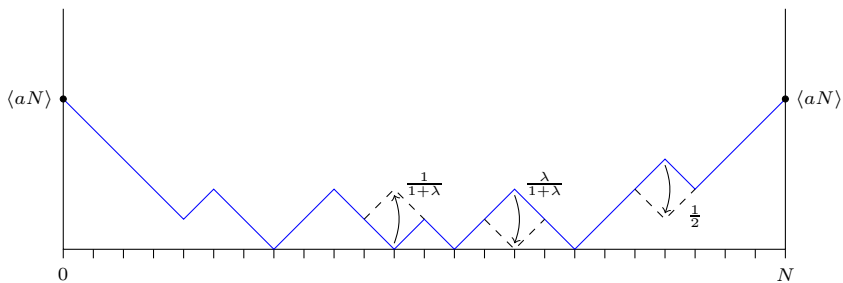
\begin{figure}[h]
\centering
  \begin{tikzpicture}[scale=.4,font=\tiny]
    \draw (26,7) -- (26,-1) -- (50,-1) -- (50,7);
    \draw[color=blue] (26,4) -- (27,3) -- (28,2) -- (29,1) -- (30,0) -- (31,1) -- (32,0) -- (33,-1) -- (34,0) -- (35,1) -- (36,0) -- (37,-1) -- (38,0) -- (39,-1) -- (40,0) -- (41,1) -- (42,0) -- (43,-1) -- (44,0) -- (45,1) -- (46,2) -- (47,1) -- (48,2) -- (49,3) -- (50,4);
    \foreach \x in {26,...,50} {\draw (\x,-1.3) -- (\x,-1);}
    \draw[fill] (26,4) circle [radius=0.1];
    \draw[fill] (50,4) circle [radius=0.1];
    \node[left] at (26,4) {$\langle aN \rangle$};
    \node[right] at (50,4) {$\langle aN \rangle$};
    \node[below] at (26,-1.3) {$0$};
    \node[below] at (50,-1.3) {$N$};
    \draw[dashed] (45,1) -- (46,0) -- (47,1);
    \draw (46,1.8) edge[out=290, in=70, ->] (46,0.2);
    \node[below] at (47,.8) {$\tfrac{1}{2}$};
    \draw[dashed] (40,0) -- (41,-1) -- (42,0);
    \draw (41,.8) edge[out=290, in=70, ->] (41,-.8);
    \node[above] at (42,.4) {$\tfrac{\lambda}{1 + \lambda}$};
    \draw[dashed] (36,0) -- (37,1) -- (38,0);
    \draw (37,-.8) edge[out=70, in=290, ->] (37,.8);
    \node[above] at (38,.4) {$\tfrac{1}{1 + \lambda}$};
  \end{tikzpicture}
 \caption{\label{fig:jumprates} Graphical representation of the jump rates for the system with elevated boundary conditions.
A transition of the chain corresponds to flipping a corner, the rate of a given transition depends on its effect on the
number of contact with the wall (note that not all possible transition are represented on the figure).
These rates are chosen so that  $\pi^\lambda_\mathcal{S}$ is reversible.}
\end{figure}

We observe that $\mathcal{L}^\lambda_\mathcal{S}$ is reversible with respect to the probability measure
\begin{equation}
 \pi^\lambda_\mathcal{S}(\eta) = \frac 1{Z_\mathcal{S}} \lambda^{-H(\eta)},
\end{equation}
where $Z_\mathcal{S} = \sum_{\eta \in \mathcal{S}} \lambda^{-H(\eta)}$. Moreover the dynamics is irreducible; therefore its semi-group converges towards $\pi^\lambda_\mathcal{S}$ as $t$ goes to infinity.
The Dirichlet form for the dynamics is defined by

\begin{equation}
\mathcal E( f)= -\sum_{\eta \in \mathcal S} (\mathcal L f) (\eta) f(\eta)\pi(\eta)=\frac{1}{2}\sum_{(\eta,\eta')\in \mathcal S^2}  \left(f (\eta')-f(\eta)\right)^2\pi(\eta)r^\gl(\eta,\eta').
\end{equation}

The spectral gap of the Markov chain is the minimal positive eigenvalue of $\mathcal L$ and the relaxation time is its inverse.
It is equal to
\begin{equation}\label{trel}
\Tr(a,\gl):=\max_{f} \frac{\var_\pi (f)}{\mathcal E(f)}=\gap^{-1}(a,\gl).
\end{equation}

\subsection{Metastability on the phase space}\label{sec:metastates}

The behavior of the dynamics depends mainly on the free-energy profile of the state-space $\cS$.
Depending on the values of $a$ and $\gl$, it might look like a single well potential or present several local minima, see Figure~\ref{fig:phase_diagram}.
In the second case one should expect a metastable behavior and a relaxation time that is proportional to $\exp(N E_a)$
where $E_a$ is the renormalized activation energy, which is the free energy barrier one has to overtake to go from a local energy minimum to the lowest energy well.

It is important to realize that the barrier for the free energy may be related to an entropic bottleneck rather than a high barrier for the Hamiltonian. For example, if one starts with a free polymer (not touching the wall), then the only obstacle for it to become pinned is that there are very few paths that have a single point of contact with the wall (most paths stay at a distance $\sim \sqrt{N}$ from the pinning height $aN$). This is a consequence of the large deviation principle for the maximum of a random walk bridge.

Even though the geometry of our space is slightly more complicated than the above description suggests, we are able to transform this heuristic picture into a rigorous result.

\medskip

There are three phases to study (apart from critical curves)
\begin{itemize}
 \item [(a)] The localized phase where $\lambda\ge \lambda_c(a)$.
In that case there is an activation energy which corresponds to the entropic cost needed to bring the middle point of the polymer down to the wall, when starting from a
flat polymer whose height oscillates around $\langle aN \rangle$.
This activation energy is independent of $\lambda$ and is equal to $q(2a)$.

\item [(b)] The case where $\frac{2}{1-2a}<\lambda<\lambda_c(a)$, for which $d_\lambda>2a$.
In that case the polymer is delocalized at equilibrium, but there is a positive activation energy to go out of the pinned phase. It is equal to
\begin{equation}
\left(\tf(\lambda)\left(1-\frac{2a}{d_\lambda}\right)-\frac{2a}{d_\lambda}q(d_\lambda)\right)+2q(2a)>0.
\end{equation}
It can also give a lower-bound on the relaxation time.
\item [(c)] The case where $d_\lambda\le 2a$. In that case there is only one local minimum in the free-energy profile and hence, no activation energy. In this case the relaxation time is polynomial in $N$.
\end{itemize}

These three phases are illustrated in Figure~\ref{fig:scaling} and their properties are a consequence of the variational principle which defines the free energy \eqref{e:varp}, they can be deduced from the proof of Proposition \ref{th:freeen}.
The main object of this paper is then to show rigorously that the mixing time is of order $\exp(N E_a)$ in cases $(a)$ and $(b)$ while it behaves like a power of $N$
in case $(c)$. We also want to show that in cases $(a)$ and $(b)$ the system has a metastable behavior in the sense that
the time to jump from the metastable state to the equilibrium state scales as an exponential random variable.

All the above statements are made precise in the following section.

\subsection{The main results}

\begin{theorem}\label{th: mainresult1}
There exists a constant $K$ such that, when $\gl\le \frac{2}{1-2a}$, for all $N$ sufficiently large
\begin{equation}
\Tr\le N^K.
\end{equation}
On the other hand when $\gl> \frac{2}{1-2a}$
\begin{equation}
\lim_{N\to \infty} \frac{1}{N}\log \Tr = E(a,\gl),
\end{equation}
where $E(a,\gl)$ is the activation energy of the system which is equal to
\begin{equation}
  \label{e:activ_E}
  E_{a,\gl}:=\begin{cases} \tf(\gl)- a \log\left(\frac{1+d_\gl}{1-d_\gl}\right)+ q (2a) &\text{ when }  \gl\in (\frac{2}{1-2a},\gl_c(a)],\\
  q(2a) & \text{ when }  \gl \in [\gl_c(a),\infty).
\end{cases}
\end{equation}
\end{theorem}

The above results culminate in the following statement, which confirms the metastable behavior of our polymer model.

We partition the state space $\cS^a_N$ into two subsets: $\bar \cS^a_N$ the set of paths that never touch the wall and
$\check{\mathcal{S}}^a_N$ the set of paths that have at least one contact point with the substrate at zero.
\begin{equation}\label{e:partition}
\begin{split}
\bar \cS^a_N&:=\{\eta\in \cS^a_N\  | \ \forall x\in [0,N],\  \eta_x>0\},\\
\check \cS^a_N&:= \cS^a_N \setminus \bar \cS^a_N= \{\eta\in \cS^a_N\  | \ \exists x\in [0,N],\  \eta_x=0\}.
\end{split}
\end{equation}
We chose the accents in $\bar \cS$ and $\check \cS$ to mimic the shape of the free and pinned polymers respectively.

Recall the definition \eqref{trel} of the relaxation time.

\begin{theorem}\label{th: mainresult2}
Fix $a \in (0,1/2)$ and $\lambda > 2/(1 - 2a)$ then set
\begin{equation}
  \mathcal{E}^N_1 =
  \begin{cases}
    & \bar S^a_N, \text{ if $\lambda \in (2/(1-2a), \lambda_c(a))$ and}\\
    & \check S^a_N, \text{ if $\lambda \ge \lambda_c(a)$}
  \end{cases}
\end{equation}
and $\mathcal{E}^N_2 = (\mathcal{E}^N_1)^c$. We then have that
\begin{equation}
  \label{e:finite_dim}
  \bP_{\pi_{\mathcal{E}^N_1}} \big[ X_{t{\Tr}} \in \mathcal{E}^N_1 \big] \xrightarrow[N \to \infty]{} \exp\{-t\}.
\end{equation}
More than that, the finite-dimensional distributions of $\1_{\mathcal{E}_1^N}(X_{t{\Tr}})$ converge to those of a process $X_t$, which starts at one and jumps with rate one to zero, where it is absorbed.
\end{theorem}

\section{Technical Preliminaries}\label{sec:prel}

\subsection{Proof of Proposition \ref{th:freeen} and sharp estimates for partition functions}

In this Section, we prove not only Proposition \ref{th:freeen}, but also a variety of precise estimate concerning the partition function of
system with further restrictions. While these estimates are sharper than what is needed to prove the existence of the free-energy,
they will be useful in the next sections when we study the dynamics and to prove Theorem \ref{th:scaling}. We focus on Equation \eqref{e:varp}, as all the other statement of Proposition \ref{th:freeen}
can be deduced from it by simple computations.

\medskip

Recalling \eqref{e:partition}, we let $\bar Z^a_N$ resp. $\check Z^a_{N}$ be  the partition function obtained by summing over the subsets $\bar S^a_N$ and $\check S^a_N$ respectively,
\begin{equation}
\bar Z^{\lambda,a}_N:=\sum_{\eta \in \bar \cS^a_N} \lambda^{H(\eta)} \quad \text{and} \quad \check  Z^{\lambda,a}_N:=\sum_{\eta \in \check{\mathcal S}^a_N} \lambda^{H(\eta)},
\end{equation}
we have
\begin{equation}
\tf(\gl,a)=\max \left(\lim_{N\to \infty} \frac{1}{N}\log \bar Z^a_N, \lim_{N\to \infty} \frac{1}{N}\log \check Z^a_{N}   \right)-\log 2,
\end{equation}
provided that these limits exist.
Equation \eqref{e:varp} is a consequence of the following result
\begin{lemma}\label{th:pineunpin}
We have
\begin{equation}\begin{split}
\lim_{N\to \infty} \frac{1}{N}\log \bar Z^a_N&=\log 2,\\
\lim_{N\to \infty} \frac{1}{N}\log \check Z^a_{N}&=\log 2+ \max_{d\in [2a,1]} \Big(\tf(\lambda) \big (1-\tfrac{2a}{d} \big )-\tfrac{2a}{d}q(d) \Big).
\end{split}
\end{equation}
Furthermore
\begin{equation}
  \label{e:max_d}
\max_{d\in [2a,1]} \Big(\tf(\lambda) \big (1-\tfrac{2a}{d} \big )-\tfrac{2a}{d}q(d) \Big)=\begin{cases} \tf(\gl)- a \log\left(\frac{1+d_\gl}{1-d_\gl}\right) &\text{  when  } \gl\ge \frac{2}{1-2a},\\
-q(2a) &\text{ when }  \gl\le \frac{2}{1-2a}.
\end{cases}
\end{equation}
\end{lemma}

\begin{proof}
The first equality is straightforward:
by standard properties of the simple random walk
\begin{equation}\label{ferte}
\bar Z^a_N=2^N \bP[S_N=0\ ; \; S_n> -\langle aN \rangle, \; \forall n\le N ] \approx  \frac{c}{\sqrt N} 2^N.
\end{equation}
To estimate the other one, we partition $\check\cS^a_{N}$ 	according to the values taken by the leftmost and rightmost point of contact with the wall.
We define for $\eta \in \bar \cS^a_N$
\begin{equation}
 L_\eta =: \inf \big\{x \in [0,N]| \eta_x = 0 \big\}, \quad R_\eta:= \sup \big\{x \in [0,N]| \eta_x = 0 \big\}.
\end{equation}
Note that these variables can take values in the set
\begin{equation}
 M^a_N = \big \{ (l,r) \in (2\mathbb{Z})^2| \langle aN\rangle \leq l \leq r \leq N - \langle aN \rangle \big\},
\end{equation}
see Figure~\ref{fig:pinned}.
Then for $(l,r)\in  M^a_N$ we define
\begin{equation}\begin{split}\label{e:defslr}
\check \cS^{l,r}&:=\{\eta\in\check{\mathcal{S}}^a_N \ | \  (L_\eta,R_\eta)=(l,r)\}, \quad (l,r) \in   M^a_N,\\
\check Z^{l,r}&:=\sum_{\eta\in \check \cS^{l,r}} \gl^{H(\eta)}.
\end{split}
\end{equation}
In words, $\check Z^{l,r}$ is the partition function of the system restricted to $\check \cS^{l,r}$.
As the cardinal of $M^a_N$ is sub-exponential in $N$, we have

\begin{equation}
\lim_{N\to \infty} \frac{1}{N}\log \check Z^a_{N}=\lim_{N\to \infty} \max_{(l,r)\in M^a_N}  \left(\frac{1}{N}  \log \check Z^{l,r}\right),
\end{equation}
should these limits exist.
Thus, our job is to control the behavior on the exponential scale of $\check Z^{l,r}$.

\begin{figure}[h]
\centering
  \begin{tikzpicture}[scale=.25,font=\tiny]
    \draw (26,10) -- (26,-1) -- (50,-1) -- (50,10);
    \draw (26,4) -- (27,3) -- (28,2) -- (29,1) -- (30,0) -- (31,1) -- (32,0) -- (33,-1) -- (34,0) -- (35,1) -- (36,0) -- (37,-1) -- (38,0) -- (39,-1) -- (40,0) -- (41,1) -- (42,0) -- (43,-1) -- (44,0) -- (45,1) -- (46,2) -- (47,1) -- (48,2) -- (49,3) -- (50,4);
    \foreach \x in {26,...,50} {\draw (\x,-1.3) -- (\x,-1);}
    \draw[fill] (26,4) circle [radius=0.1];
    \draw[fill] (50,4) circle [radius=0.1];
    \node[left] at (26,4) {$\langle aN \rangle$};
    \node[below] at (26,-1) {$0$};
    \node[below] at (50,-1) {$N$};
    \node[right] at (50,4) {$\langle aN \rangle$};
    \node[above] at (55,12) {$L_\eta$};
    \node[right] at (69,-2) {$R_\eta$};
    \node[below] at (55,-2) {$\langle aN \rangle$};
    \node[below] at (33,-1) {$L_\eta$};
    \node[below] at (43,-1) {$R_\eta$};
    \node[below] at (69,-2) {$N - \langle aN \rangle$};
    \draw (55,-2) rectangle (69,12);
    \draw (67,0) circle [radius=.3];
    \foreach \x in {55,...,69} {
      \pgfmathparse{\x - 57}
      \xdef\z{\pgfmathresult}
      \foreach \y in {-2,...,\z} {\draw[fill] (\x,\y) circle [radius=.08];}}
  \end{tikzpicture}
  \caption{A surface in $\mathcal S^{0.24}_{24,(7,17)}$ (left) and the corresponding set $M^{0.24}_{24}$.}
  \label{fig:pinned}
\end{figure}
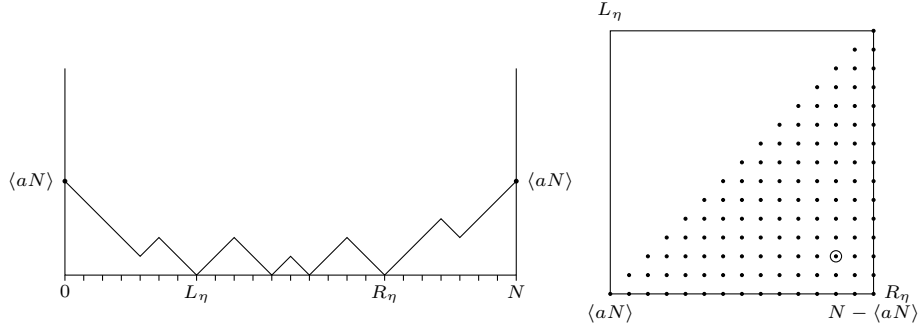


\begin{lemma}\label{th:retl}
We have for all $l,r\in M^a_N$
\begin{equation}\label{e:presk}
\check Z^{l,r}= (1+o(1))\frac{\langle aN\rangle}{l}\frac{\langle aN\rangle}{N-r}\binom{(l+\langle aN\rangle)/2}{l}\binom{(N-r-\langle aN\rangle)/2}{N-r}Z^\gl_{r-l}.
\end{equation}
where the $o(1)$  term tends to zero when $N\to \infty$ uniformly in $l,r$.
As a consequence there exist a constant $C$ (depending on $\gl$, and $a$) such that for all $N, r$ and $l$,
\begin{equation}\label{e:grossomodo}
\frac 1 C  Y(N,l,r) \le  \check Z^{l,r}\le C Y(N,l,r),
 \end{equation}
where
\begin{align}\label{e:grossemoto}
Y(N,l,r)= & 2^N \sqrt{\frac{1}{l-\langle aN\rangle+1}}\sqrt{\frac{1}{N-r-\langle aN\rangle+1}}\\
& \times \exp\big\{ F(\gl)(r-l)- l q(\langle aN \rangle/ l)- (N-r) q(\langle aN \rangle/ (N-r)) \big\}.
\end{align}
\end{lemma}

We finish the proof of Lemma \ref{th:pineunpin} and postpone the proof of Lemma \ref{th:retl}

To shorten the expressions involved in the above lemma, we write
\begin{equation}
  \bar Y(N,l,r) := \exp\left( F(\gl)(r-l)- l q(\langle aN \rangle/ l)- (N-r) q(\langle aN \rangle/ (N-r)) \right).
\end{equation}

Using Equation \eqref{e:grossomodo}, we have
\begin{multline}
 \max_{(l,r)\in M^a_N} \frac{1}{N}  \log \check Z^{l,r}=\\
 \max_{(l,r)\in M^a_N}  \left(\tf(\gl)\frac{r-l}{N}- \frac{l}{N} q(\langle aN \rangle/ l)- \frac{(N-r)}{N} q(\langle aN \rangle/ (N-r))\right)+\log(2)+o(1).
  \end{multline}
Considering $d_1=\langle aN \rangle/ l$, $d_2=\langle aN \rangle/ (N-r))$, it is a standard exercise to show that the limit of the above maximum is
\begin{equation}
\max_{\{d_1, d_2\le 1  \ |  \ a( d_1^{-1}+d_2^{-1})\le 1\}} \left( \tf(\gl)\big(1-a( d_1^{-1}+d_2^{-1})\big)- a\big(d_1^{-1} q(d_1)+ d_2^{-1} q(d_2)\big)\right).
\end{equation}
Then we conclude \eqref{e:max_d} by remarking that the maximum above is attained for $d_1=d_2$.
\end{proof}

\begin{proof}[Proof of Lemma \ref{th:retl}]
Because the end points of the pinned region are fixed, the set $\check \cS^{l,r}$ has a natural product structure which yields
\begin{equation}\label{e:defQ}
\check Z^{l,r}=Z^\gl_{r-l}Q(l,\langle aN \rangle)Q(N-r,\langle aN \rangle)
\end{equation}
where
$$Q(a,b):=\# \{(S_n)_{n\in [0,a]}\  | \ S_0=0, \  S_a=b, \forall n\in [1,a],\ S_{n}>0 \text{ and }
|S_n-S_{n-1}|=1 \}.$$
Without the constraint $S_{n}>0$, $Q(a,b)$ would be just a binomial coefficient.
The constraint just yields a factor
$$\bP\left[S_n>0, \ \forall n\in (0,a] \ | \ S_a=b \right],$$ which is controlled with Lemma \ref{th:constraint}, yielding \eqref{e:presk}.

\medskip

To deduce \eqref{e:grossomodo} from it, we use the Stirling formula to estimate the binomial coefficients and \cite[Theorem 2.2 (1)]{Gia07}  to replace
$Z^\gl_{r-l}$ by $\exp((F(\gl)+\log 2)(r-l))$.
\end{proof}

\subsection{Typical behavior at the critical point $\gl=\gl_c(a)$}

As a consequence of the estimates of the previous section, we can prove that the polymer is typically pinned when
$\gl=\gl_c(a)$.

\begin{proposition}\label{critbeha}
For any $a\in(0,1/2)$, there exists a constant $C$ such that for every $N\ge 0$
\begin{equation}
\frac{1}{C}2^N\le \check Z_N^{\gl_c(a)}\le C 2^N.
\end{equation}
As a consequence
\begin{equation}\label{crimou}
\lim_{N\to \infty} \pi^{\gl_c(a)}(\check \cS^a_N)=1.
\end{equation}
\end{proposition}

\begin{proof}
From Lemma \ref{th:retl} it is sufficient to show that when $\gl=\gl_c(a)$
\begin{equation}\label{youple}
\sum_{r,l} \sqrt{\frac{1}{l-\langle aN\rangle+1}} \sqrt{\frac{1}{N-r-\langle aN\rangle+1}}\bar Y(l,r)
\end{equation}
is bounded away from zero and infinity. If one allows $r$ and $l$ to assume real values, then the quantity in the exponent of $\bar Y(l,r)$, that is
$$y(l,r):=(r-l)\tf(\gl)-lq(\langle aN \rangle/ l)-(N-r)q(\langle aN \rangle /(N-r)),$$
is maximized at $(l_{\max},r_{\max})$, where $l_{\max}=N-r_{\max}=\frac{(\gl+1)\langle aN \rangle}{\gl-1}$.

\medskip

Hence we can restrict the sum in \eqref{youple} to
$$ l\in \left[ \left(\frac{(\gl+1)a}{\gl-1}-\gep  \right)N,  \left(\frac{(\gl+1)a}{\gl-1}+\gep  \right)N\right],$$
 $$r \in \left[ \left(1-\frac{(\gl+1) a}{\gl-1}-\gep  \right)N,  \left(1-\frac{(\gl+1)a}{\gl-1}+\gep  \right)N\right].$$
For a fixed small $\gep>0$ as the rest of the sum gives a contribution which is exponentially small in $N$.
In this interval the square root term in \eqref{youple} is always of order $N^{-1}$.
Hence what remains to show is that
$$\sum_{l\in \left[ \left(\frac{(\gl+1) a}{\gl-1}-\gep  \right)N,  \left(\frac{(\gl+1) a}{\gl-1}+\gep  \right)N\right] \cap 2\bbZ \;\; }
\sum_{r\in \left[ \left(1-\frac{(\gl+1)a}{\gl-1}-\gep  \right)N, \left(1-\frac{(\gl+1) a}{\gl-1}+\gep  \right)N\right] \cap 2 \mathbb{Z}}
\bar Y(l,r),$$
is of order $N$.

\medskip

Using the second order Taylor expansion (in $r$ and $l$) of $y$ around the maximal points (the reader can check that the second derivative is of order $1/N$), we can find a positive constant $C$ such that uniformly in $N$
$$\frac{1}{C}\left[(l-l_{\max})^2+(r-r_{\max})^2\right] N^{-1}  \le y(l,r)-y(l_{\max},r_{\max}) \le C \left[(l-l_{\max})^2+(r-r_{\max})^2\right] N^{-1}.$$
Combining this with the fact that $y(l_{\max},r_{\max})=O(1)$ (and this is the only place where $\gl=\gl_c$ is needed), and usual results to compare sums and integrals, we obtain the desired result.

\medskip

To deduce \eqref{crimou} we use \eqref{ferte}.
\end{proof}

\subsection{Proof of Theorem \ref{th:scaling}}
\label{ss:scaling_limit}

We start by proving \eqref{scal2}. Since the scaling limit when $\gl < \gl_c(a)$, is trivial in the absence of wall, the only thing left to check is that
\begin{equation}
\lim_{N\to \infty} \pi^{\gl}(\check \cS^a_N)=0.
\end{equation}
This is true by Lemma~\ref{th:pineunpin} and in fact the convergence is exponentially fast.

\medskip

Let us now move to the proof of \eqref{scal1}.
Thanks to Proposition \ref{critbeha}, we know that when $\gl\ge \gl_c(a)$ the polymer is typically pinned. Hence it is sufficient to prove the result for the restricted measure
$\pi( \cdot | \check \cS^a_N)=: \check \pi$. Set $$L_0:=N a\gl/(\gl-2) \quad \text{and} \quad R_0:=N(1- a\gl/(\gl-2)).$$
Using Lemma~\ref{th:retl}, it is possible to check (we leave it as an exercise) that

\begin{equation}
 \lim_{N\to \infty}\pi^{\gl_c(a)}\Big(L_{\eta}\in \big(L_0-N^{\alpha},L_0+N^{\alpha}\big), R_{\eta}\in \big(R_0-N^{\alpha},R_0+N^{\alpha} \big) \Big)=1,
\end{equation}
for some $\alpha \in(1/2,1)$.

What remains to check then is that conditionally to any value of $L_\eta$ and $R_\eta$ in the interval given above, the probability appearing in \eqref{scal1} decays to zero
(in fact it decays to zero exponentially fast). This is very standard and we also leave it as an exercise to the reader.

\section{Dynamics}

In what follows we analyze the dynamics introduced in Subsection~\ref{sec:dyna}. We first prove lower and upper bounds on the relaxation time of the dynamics on different phases.

\subsection{Activation energy and lower bound on the relaxation time}

In this section, we use the equilibrium estimates proved in Section \ref{sec:prel} to get an exponential lower-bound on the relaxation time of the dynamics for the $\lambda > 2/(1-2a)$ phase.
The idea is just to localize a bottleneck in the space of polymer configuration (see \cite[Section 13.3]{LPW09}).

\begin{proposition}
\label{p:activ_energy}
When $\gl>\frac{2}{1-2a}$, one has
\begin{equation}
\liminf_{N\to \infty} \frac{1}{N}\log \Tr \ge E(a,\gl),
\end{equation}
where $E(a,\gl)$ is defined in  \eqref{e:activ_E}.
\end{proposition}

When $\gl>\frac{2}{1-2a}$, suppose that the polymer starts from an initial configuration that is metastable (either the pinned state when $\gl<\gl_c(a)$ or the unpinned one when $\gl > \gl_c(a)$). Then, in order to attain equilibrium, it has to visit a configuration that has only one contact with the wall. These configurations are difficult to reach, since they are very few (when compared to the free phase) and they don't present a substantial energy compensation (as the pinned phase would). In other words, they represent an entropic bottleneck.

\begin{proof}
We use the characterization \eqref{trel} of the relaxation time to obtain the lower bound with the function
 $f(\eta):=\ind_{\check{\mathcal{S}}^a_N}$ the indicator function of trajectories with at least one contact with the wall.
We have

$$\pi^{\gl}(\check{\mathcal{S}}^a_N):=\frac{\check{Z}^a_N}{Z^a_N}.$$
and thus
$$\var_\pi(f)=\pi^{\gl,a}_N(\check{\mathcal{S}}^a_N)-\left(\pi^{\gl,a}_N(\check{\mathcal{S}}^a_N)\right)^2=\frac{\check{Z}^a_N  \bar Z^a_N}{(Z^a_N)^2}.$$

One can compute explicitly the Dirichlet form of $f$. It is equal to
\begin{equation}
\cE(f)=\frac{1}{1+\gl}\pi^{\gl}(\partial \check{\mathcal{S}}^a_N),
\end{equation}
where for a set $\mathcal A\subset \mathcal{S}^a_N$,
$$ \partial {\mathcal{A}} := \{ \eta \in \cA \ | \ \exists \eta' \in \mathcal{S}^a_N\setminus \cA, r^\gl(\eta,\eta')>0\}.$$
In particular $\partial \check{\mathcal{S}}^a_N$ is the set of polymers with exactly one contact with the wall. We have
$$\pi^\gl(\partial \check{\mathcal{S}}^a_N)=\frac{\left(\sum_{l=\langle aN \rangle}^{N-\langle aN \rangle} \check Z^{l,l}\right)}{Z^a_N}.$$
Hence
\begin{equation}
\frac{\var_\pi(f)}{\cE(f)}=(1+\gl)\frac{\check{Z}^a_N \bar Z^a_N}{Z^a_N\left(\sum_{l=\langle aN \rangle}^{N-\langle aN \rangle} \check Z^{l,l}\right)}.
\end{equation}
Using the same tools as in the proof of Lemma \ref{th:pineunpin}, we obtain that
\begin{multline}\label{eq:boundar}
\lim_{N\to \infty} \frac{1}{N} \log \left(\sum_{l=\langle aN \rangle}^{N-\langle aN \rangle} \check Z^{l,l}\right)\\
=\max_{\{ x\in [a,1-a]\}} \left( x q(x^{-1}a)+ (1-x)q((1-x)^{-1}a)\right)=\log 2-q(2a).
\end{multline}
Thus using what we know about the asymptotic behavior of $\check{Z}^a_N$  $\bar Z^a_N$ and $Z^a_N$ (Proposition~\ref{th:freeen} and Lemma \ref{th:pineunpin})
we have
\begin{equation}
\lim_{N\to \infty} \frac{1}{N}\log \left(\frac{\var_\pi(f)}{\cE(f)}\right)=\begin{cases}\tf(\gl)- a \log\left(\frac{1+d_\gl}{1-d_\gl}\right)+q (2a) &\text{ when }  \gl\in (\frac{2}{1-2a},\gl_c(a)],\\
 q(2a) \quad  \text{ when }  \gl \in [\gl_c(a),\infty).
\end{cases}
\end{equation}
which, thanks to the characterization \eqref{trel} of the relaxation time, yields the result.
\end{proof}

\subsection{Decomposition of Markov chains}
\label{s:decomp}
Whereas deriving a lower bound on the relaxation time is quite straightforward once the bottleneck has been properly identified, upper-bounds usually require more work. Our strategy here is to decompose our Markov chain into smaller chains for which we are able to compute the mixing time, using either the flux method (Lemma \ref{technos}) or well established results for dynamical pinning with zero boundary condition (from \cite{CMT08}).

In the present section, we quote the main result that allows for such decomposition. Roughly speaking, this is a continuous time version of the estimates developed in \cite{JSTV04}, which provide a way to estimate the spectral gap of a chain in terms of the gaps of its decomposed parts.

\medskip

For this, let $\mathcal S$ be a state space endowed with a collection of transition rates $r(\eta, \eta')$ inducing a Markov process on $\mathcal{S}$ which is reversible with respect to some probability measure $\pi$.


Assume that $\mathcal{S}$ is partitioned into a disjoint union of subset $\{\mathcal S_i\}_{i \in \bar{M}}$ and let $\bar \pi$ be the probability measure on $\bar{M}$ induced by $\pi$ which is defined by
\begin{equation}
  \bar \pi(i)= \sum_{\eta \in \mathcal S_i} \pi (\eta), \text{ for $i \in \bar{M}$}.
\end{equation}
Define the generator $\bar {\mathcal L}$ acting on functions $\phi: M\mapsto \bbR$ by
\begin{equation}
(\bar \cL \phi)(i)=\sum  \bar r(i,i')(\phi(i')-\phi(i))
\end{equation}
where
\begin{equation}
 \label{e:projection}
 \bar r(i,i') = \left(\bar \pi(i)\right)^{-1} \sum_{\eta \in \mathcal S_i, \eta' \in \mathcal S_{i'}} \pi(\eta) r(\eta, \eta').
\end{equation}
This defines a continuous time Markov process on $\bar{M}$ that is reversible with respect to $\bar \pi$.

We also introduce, for each $i\in M$, the restricted chain in $\mathcal S_i$ as which corresponds to the original chain, with the transitions that exit $\mathcal{S}_i$ are canceled.
Its generator is given by
\begin{equation}
\label{e:restrict}
( \cL_i f)(\eta)=\sum_{\eta'\in \mathcal S_i} \bar r_i(\eta,\eta')(f(\eta')-f(\eta)).
\end{equation}
This induces a Markov process which is irreducible w.r.t\ $\pi_i:= \pi( \cdot | \mathcal S_i)$.
We set $\Gap$, $(\Gap_i)_{i\in \bar M}$ and $\overline{\Gap}$ to be the spectral gap associated to $\cL$, $(\cL_i)_{i \in \bar M}$ and $\bar \cL$ respectively.

\medskip

As remarked in \cite[Proposition 2.1]{CLMST12},
the following adaptation of \cite[Theorem 1]{JSTV04} in continuous time holds
 \begin{proposition}\label{t:gap}
 Set $\gamma:=\max_i\max_{\eta\in  \cS_i}\left(\sum_{\eta' \in \mathcal S \setminus \mathcal S_i} r(\eta,\eta')\right)$,
then
 \begin{equation}
  \Gap \ge \min \left( \frac{\bar \gap}{3} , \frac{\bar \gap(\min_{i\in M} \gap_{i})}{\bar \gap +\gamma}\right).
 \end{equation}
\end{proposition}

In what follows, we will make repeated use of the above proposition in order to estimate the relaxation time of the process from above.

\subsection{Upper-bound on the relaxation time}

\begin{proposition}\label{upbound}
There exists a constant $K$ such that when $\gl\le \frac{2}{1-2a}$, for all $N$ sufficiently large
\begin{equation}\label{eq:poli}
\Tr\le N^{16}.
\end{equation}
On the other hand when $\gl> \frac{2}{1-2a}$
\begin{equation}
\limsup_{N\to \infty} \frac{1}{N}\log \Tr \le  E(a,\gl),
\end{equation}
\end{proposition}

An important step for the proof is to show that the chain restricted to each of the local wells $\check{\mathcal{S}}^a_N$ and
$\bar {\mathcal{S}}^a_N$, mixes rapidly. For the unpinned phase, this is an easy consequence of \cite{Wil04}. While for the pinned phase this is much more delicate and will be proved it in Section \ref{fastmix}


\begin{theorem}
 \label{t:pinned}
 Fix any $a \in (0,1/2)$ and $\lambda \geq 0$ and consider the process on the pinned phase $\check{\mathcal{S}}^a_N$, derived from the original rates in \eqref{e:gen} after the restriction introduced in \eqref{e:restrict}. Then,
 \begin{equation}
  \label{e:pinned}
  \Gap_1 \geq c(a,\lambda) N^{-12},
 \end{equation}
 for every $N \geq 1$.
\end{theorem}

\begin{remark} The powers of $N$ that are present in Proposition \ref{upbound} and Theorem \ref{t:pinned} are far from being optimal. It is reasonable to think that
the spectral gap should be of order $L^{-2}$ when $\gl<2/(1-2a)$ but we are not able to prove this with our method. An interesting issue would be to determine whether there is a critical slow-down of the dynamics: if the spectral gap becomes much smaller when $\gl = 2/(1-2a)$.
\end{remark}

\begin{proof}[Proof of Proposition \ref{upbound}]
We use Proposition \ref{t:gap}, for the decomposition of the chain in the two subspaces $\cS_1:=\check{\mathcal{S}}^a_N$ and $\cS_2:=\bar \cS^a_N$, from Theorem \ref{t:pinned} the relaxation time for the chain restricted to $\cS_1$ is smaller or equal to $c(a,\gl) N^{12}$, and according to Proposition \ref{th:cornerflip} (proved in the Appendix) the relaxation time for the chain restricted to $\cS_2$ is $O(N^2)$. Furthermore $\gamma$ is bounded above by $N$.

Hence the important thing is to control $\overline{\gap}$, i.e. \ to control the rate of jump from one phase to the other.

\medskip

We note that
\begin{equation}
\bar r(1,2)=\frac{1}{1+\gl}\frac{\pi(\partial \check{\mathcal{S}}^a_N)}{\pi( \check{\mathcal{S}}^a_N)}
\end{equation}
and thus
\begin{equation}
  \label{e:bar_gap_rate}
  \overline{\gap}=\frac{\bar r(1,2)}{\bar \pi (2)}=\frac{1}{1+\gl}\frac{Z^a_N\left(\sum_{l=\langle aN \rangle}^{N-\langle aN \rangle} \check Z^{l,l}\right)}{\bar Z^a_N \check Z^a_N}.
\end{equation}
We deduce from \eqref{eq:boundar}, Lemma \ref{th:pineunpin} and Proposition \ref{th:freeen}, that
\begin{equation}
-\lim_{N\to \infty} \frac{1}{N}\log \overline{\gap}:=\begin{cases} 0 &\text{ if } \gl \le \frac{2}{1-2a},\\
E(a,\gl) &\text{ if } \gl \ge \frac{2}{1-2a}.
\end{cases}
\end{equation}
This concludes the case $\gl > \frac{2}{1-2a}$.

\medskip

To prove the inequality \eqref{eq:poli} we need to show that $\overline{\gap}$ is bounded from below by an appropriate power of $N$. From \eqref{e:bar_gap_rate}, using  the Lemma~\ref{l:one_touch} (see below) we have
\begin{equation}
  \overline{\gap} \geq \frac{1}{\gl+1} \bar r(1,2)\ge c(a,\gl) N^{-3}
\end{equation}

We finish the proof using Proposition~\ref{t:gap} to conclude that
\begin{equation}
  \begin{split}
    \gap & \geq c(\gl) \min \Big( \frac{\overline{\gap}}3, \frac{\overline{\gap} \min(\gap_1, \gap_2)}{\overline{\gap} + \gamma} \Big)\\
    & \geq c(a,\gl) \min \Big( N^{-3}, \frac{N^{-3 - 12}}{N} \Big) \geq c(a,\gl) N^{-16},
  \end{split}
\end{equation}
finishing the proof of the proposition modulo Lemma~\ref{l:one_touch} which is established below.
\end{proof}

\begin{lemma}
\label{l:one_touch}
 When $\gl \le  \frac{2}{1-2a}$, there exists $C = C(a,\lambda)>0$ such that
 \begin{equation}
 \frac{ \pi(\partial \check{\mathcal{S}}^a_N)}{\pi( \check{\mathcal{S}}^a_N)}\ge C N^{-3}.
 \end{equation}
\end{lemma}

\begin{proof}
We just have to prove that
$$\max_{l\in \{\langle aN \rangle,\langle aN \rangle+2,\dots, N-\langle aN\rangle\} } \check Z^{l,l} \ge c(a,\gl) N^{-1}\max_{l, r \in M^a_N} \check Z^{l, r}.$$
And according to \eqref{e:grossomodo}, it is enough to show that
\begin{equation}
  \max_{l\in \{\langle aN \rangle,\langle aN \rangle+2,\dots, N-\langle aN\rangle\} } \bar Y (N,l,l) \geq c(a, \lambda) \max_{l, r \in M^a_N} \bar Y(N,l,r).
\end{equation}

Then, writing the inequality for the $\log$, it is sufficient to check that there exists $C>0$,
\begin{multline}\label{e:maximizer}
\max_{l\in \{\langle aN \rangle,\langle aN \rangle+2,\dots, N-\langle aN\rangle\} } -l q(\langle aN \rangle/ l)- (N-l) q(\langle aN \rangle/ (N-l))\\
\ge \max_{(l,r)\in M^a_N} \tf(\gl)(r-l) - l q(\langle aN \rangle/ l)- (N-r) q(\langle aN \rangle/ (N-r))-C.
\end{multline}
Should we not have the restriction that $r$ and $l$ are even integers, the l.h.s. of
Equation \eqref{e:maximizer} would be maximized when $l=N/2$ and the r.h.s.\ when
$$l=N-r:= \min\left( \frac{\langle aN \rangle}{d_\gl},N/2\right),$$
which equals $N/2$ when $\gl \le  \frac{2}{1-2a}$.
This proves \eqref{e:maximizer} if $N/2$ is even.
If $N/2$ is not an even integer, taking $N/2+1$ instead of $N/2$ only changes the value of
$$ -l q(\langle aN \rangle/ l)-(l-N) q(\langle aN \rangle/ (N-l))$$ by a constant amount and thus  \eqref{e:maximizer} holds for a
well chosen constant $C$.
\end{proof}

\subsection{Proof of Theorem \ref{t:pinned}}\label{fastmix}

In this section we establish that the pinned phase of the polymer dynamics (see Figure~\ref{fig:scaling}) mixes fast, uniformly in the parameters $a \in (0,1/2)$ and $\lambda \in \mathbb{R}$.
We do so by using Proposition \ref{t:gap} for an appropriate decomposition of the dynamics in the pinned phase.
We use the partition (recall \eqref{e:defslr})
$$\check{\cS}^a_N=\bigcup_{(l,r)\in M^a_N} \check{\mathcal{S}}^{l,r}.$$
We call $\gap(l,r)$ the spectral gap of the dynamics $\eta^{l,r}$ restricted to $\check{\mathcal{S}}^{l,r}$,
$\overline{\gap}$ the gap of dynamics $X$ defined on $M^a_N$ by \eqref{e:projection}.
We can easily prove Theorem \ref{t:pinned} using the following result

\begin{proposition}\label{t:estimates}
The three following bounds hold for all $N$.
\begin{itemize}
\item [(i)]
$$\min_{(l,r)\in M^a_N} \gap(l,r)\ge (1-\cos \pi/N).$$
\item [(ii)]
$$\max_{(l,r)\in M^a_N} \max_{\eta\in \check{\mathcal{S}}^{l,r}}\sum_{ \eta' \in \check{\cS}^a_N\setminus  \check{\mathcal{S}}^{l,r}} r^\gl(\eta,\eta')\le N.$$
\item [(iii)]
$$\overline{\gap} \le C N^{-10}$$
\end{itemize}

\end{proposition}

\begin{proof}[Proof of  Proposition \ref{t:estimates} $(i)$ and $(ii)$]
For the first point we show that $\eta^{l,r}$ is composed of three independent components, corresponding to the intervals $[0,l], [l,r]$ and $[r,N]$ (the middle one being possibly degenerated when $l=r$). More precisely,
let $\{\eta^1(t)\}_{t \geq 0}$, $\{\eta^2(t)\}_{t \geq 0}$ and $\{\eta^3(t)\}_{t \geq 0}$ be given by the restrictions of $\eta^{l,r}(t)$ to $[0,l]$, $[l,r]$ and $[r,N]$ respectively;
$\eta^1(t):\{0,\dots,l\} \to \mathbb{Z}_+$ being given by $\eta^1_x(t) = \eta_x(t)$ and analogously for $\eta^2$ and $\eta^3$.
Due to the inhibition of transitions at the connection points $l$ and $r$, $\eta^1$, $\eta^2$ and $\eta^3$ are independent dynamics.
Hence for fixed $l$ and $r$ we have

\begin{equation}
 \gap(l,r)=\min(\gap_1,\gap_2,\gap_3).
\end{equation}

The dynamics $\eta^1$ and $\eta3$ are both corner-flip dynamics with constraint and enter the frame-work of Appendix~\ref{cfcs}. Hence from Proposition \ref{th:cornerflip}
$$ \min(\gap_1, \gap_3)\ge \min( (1-\cos(\pi/l), (1-\cos(\pi/(N-r))))\le 1-\cos(\pi/N).$$
Concerning $\eta_2$, \cite[Theorem 3.1]{CMT08} gives
$$ \gap_2\ge 1-\cos(\pi/(r-l))\ge 1-\cos(\pi/N).$$
This proves part (i), while (ii) is easy to check from the definition of the rates \eqref{e:defrates}.
\end{proof}

The remaining and most delicate point is to estimate the spectral gap of the projection chain ($\overline{\gap}$).
The method we use for this is to use Proposition \ref{t:gap} again to reduce the job to estimating spectral gap of a one dimensional chain.

\subsection{Proof of Proposition \ref{t:estimates} point $(iii)$}

We partition $M^a_N$ into the following disjoint subsets
\begin{equation}
 M^a_N = \bigcup_{r\in [ \langle aN \rangle, N-\langle aN \rangle ] \cap 2\bbZ} ([ \langle aN \rangle,  r ] \cap 2\bbZ) \times \{r\} =: \bigcup_{r\in [ \langle aN \rangle, N-\langle aN \rangle ] \cap 2\bbZ} M_r.
\end{equation}

For this decomposed chain we have again $\gamma\le N$. We call $\bar X$ the projection chain on $[ \langle aN \rangle, N-\langle aN \rangle ] \cap 2 \bbZ$ and $\overline{\gap}_2$ its spectral gap.
We call $X^r$ the chain reduced to $M_r$ (which we can identify with $2\bbZ\cap    [ \langle aN \rangle,  r ]$) and $\gap_r$ its spectral gap.

\begin{proposition}\label{t:pasfacile}
 There exists a constant $C(\gl)$ such that for all $N$
 \begin{itemize}
  \item [(i)] $\gap_r\le C N^{5/2}$, \text{ for $r \in (2\bbZ \cap [ \langle aN \rangle, N-\langle aN \rangle ])$}
  \item [(ii)] $\overline{\gap}_2\le C N^{5}$
 \end{itemize}
\end{proposition}

Both $X_r$ and $\bar X$  are one dimensional Markov chains. Using the method of fluxes introduced by Sinclair it is simple to control the spectral gap of such chains, as shown in the following lemma.

\begin{lemma}\label{technos}
If one has a reversible chain on $\{0,\dots,L\}$ with equilibrium measure $p$ and transitions $q$ that satisfies the two following conditions
\begin{itemize}
\item[(a)] $\min_{n\in[1,L]} \max(q(n,n-1),q(n-1,n)) \ge \alpha$
\item[(b)] For any $n$,
$$ \min (\sum_{m \le n} p(n),  \sum_{m \ge n} p(n))\le \beta p(n)$$
\end{itemize}
Then the relaxation time of the chain is smaller than $\beta L/ \alpha$.
\end{lemma}

\begin{proof}
We have to estimate the quantity $B$ of \cite[Corollary 13.24]{LPW09}, where our choice for the paths $(\Gamma_{n,m})$ consists in taking the shortest nearest neighbor paths.
For a given $e=(n-1,n)$, $Q(e):=p(n-1)q(n-1,n)=p(n)q(n,n-1)$.
The ratio for which we want an upper bound is
\begin{equation}
\frac{1}{Q(e)}\sum_{m \le n-1}\sum_{z\ge n} p(m)p(z) L\le \frac{1}{\alpha p(n)}\sum_{m\le n-1}\sum_{z\ge n} p(m)p(z) L=\frac{\beta L}{\alpha}
\end{equation}
We supposed here that $\max(q(n,n-1),q(n-1,n))=q(n,n-1)$, but it also works the other way around.
\end{proof}

\begin{proof}[Proof of Proposition \ref{t:pasfacile}]

 The plan for the proof is the same for the two points:
 prove that the assumption of Lemma \ref{technos} are satisfied some $\alpha$ and $\beta$ using the estimates of Lemma \ref{th:retl}.

\medskip

 Let us start with $X^r$. We can identify $M_r$ with $\{0,\dots,L\}$, $L=(r-\langle aN \rangle)/2$.
 To check the point $(a)$ of Lemma \ref{technos}, we remark that if $l<r$ the rate at which $X_r$ jumps from $l$ to $l+2$
 is equal to
 \begin{equation}\label{e:troidez}
  \frac{1}{1+\gl}\times \pi^{\gl,a}_N\left(\eta_{l+2}=0\ | \ L_\eta=l, R_{\eta}=r \right)=
  \frac{1}{1+\gl}\frac{\gl Z^\gl_{r-l-2}}{Z^{\gl}_{r-l}},
 \end{equation}
and is bounded from below by a constant that depends only on $\gl$ (see \cite[Theorem 2.2]{Gia07}).
The point $(b)$ is equivalent to
\begin{equation}\label{mixos_1}
 \min \left(\sum_{m \le l} \check Z^{m,r},  \sum_{m \ge l}\check Z^{m,r}\right)\le \beta \check Z^{l,r}.
\end{equation}
From \eqref{e:grossomodo}-\eqref{e:grossemoto}, it is sufficient to prove .
\begin{equation}\label{mixos}
 \min \left(\sum_{m \le l}\bar Y(N,m,r),  \sum_{m \ge l}\bar Y(N,m,r)\right)\le \frac{\beta}{C\sqrt{N}}\bar Y(N,l,r).
\end{equation}

Let us rewrite $\bar Y(N,l,r)$ to underline the dependence in $l$.

$$\bar Y(N,l,r)=:A(N,r) \exp\left( -\tf(\gl)l- l q(\langle aN \rangle/ l)\right)$$

The reader can check that the function $q$ is convex and hence that $l\mapsto -\tf(\gl)l- l q(\langle aN \rangle/ l)$ is a concave function and hence
that $$U(l):=\exp\left( -\tf(\gl)l- l q(\langle aN \rangle/ l)\right)$$ has a unique local maximum.
As a consequence, for all $l$, $U(m)\le U(l)$ either for all $m \geq l$ or for all $m \leq l$, thus
\begin{equation}\label{mixos2}
 \min (\sum_{m \le l} U(m),  \sum_{m \ge l} U(m))\le N U(l).
\end{equation}
This implies that \eqref{mixos} is satisfied for
$\beta= C N^{3/2}$, with a constant $C$ that does not depend $N$ or $r$.

\medskip

Let us now move to $\bar X$, whose state space can be identified with $\{0,\dots,L\}$, $L=(N-\langle aN \rangle)/2$.
The rate at which $\bar X$ jumps from $r$ to $r-2$
 is equal to

 \begin{multline}
  \frac{1}{1+\gl}\pi^{\gl,a}_N\left(\eta_{r-2}=0\ | \  R_{\eta}=r \right)\\=
  \frac{1}{1+\gl} \pi^{\gl,a}_N\left(\eta_{r-2}=0\ | \  R_{\eta}\
  =r, L_{\eta}\le r-2 \right)
    \times \pi^{\gl,a}(L_{\eta}\le r-2 | \  R_{\eta}=r)
 \end{multline}
From \eqref{e:troidez} and the comment below it $$\pi^{\gl,a}_N\left(\eta_{r-2}=0\ | \  R_{\eta}=r, L_{\eta}\le r-2 \right)\ge c(\gl)>0.$$
For the second factor we remark that

 \begin{multline}
 \pi^{\gl,a}(L_{\eta}\le r-2 | \  R_{\eta}=r)\ge \frac{\pi^{\gl,a}(L_{\eta}= r-2 | \  R_{\eta}=r)}{\pi^{\gl,a}(L_{\eta}= r-2 | \  R_{\eta}=r)+\pi^{\gl,a}(L_{\eta}= r | \  R_{\eta}=r)}\\
 \ge  \min\left(1/2,  \frac{\pi^{\gl,a}(L_{\eta}= r-2 | \  R_{\eta}=r)}{\pi^{\gl,a}(L_{\eta}= r | \  R_{\eta}=r)}\right),
  \end{multline}
  and that (recall \eqref{e:defQ})
  $$\frac{\pi^{\gl,a}(L_{\eta}= r-2 | \  R_{\eta}=r)}{\pi^{\gl,a}(L_{\eta}= r | \  R_{\eta}=r)}=\frac{\gl Q((r-2),\langle aN \rangle)}{Q(r,\langle aN \rangle)}\ge C N^{-1},$$
  where the last inequality is obtained by using Lemma \ref{th:constraint} to replace $Q$ by binomial coefficients.

  \medskip

 Hence we can choose $\alpha=CN^{-1}$ in Lemma \ref{technos}.
 For point $(b)$ of Lemma \ref{technos} we need to prove that
 \begin{equation}\label{mixos3}
  \min \left(\sum_{q \le r} \check Z^{q},  \sum_{q \ge l}\check Z^{q}\right)\le \beta \check Z^{r}.
 \end{equation}
 where $$\check Z^{r}=\sum_{l \le r} \check Z^{l,r}$$ is the partition function restricted to trajectories which satisfies
 $R_{\eta}=r$. Changing $\gb$ by a factor $C N$, on can replace $\check Z^r$ by

 $$V(r)= \sum_{l\le r} \bar Y(N,l,r).$$

 Then let us call $l_{\max}$ the point where the maximal value of $ -\tf(\gl)l- l q(\langle aN \rangle/ l)$ is reached.
Then as $l \mapsto  \bar Y(N,l,r)$ is unimodal, its maximal value on $\{ l \ | l\le r\}$ is reached at $\min(l_{\max},r)$
and hence
\begin{equation}\label{groom}
1\le  V(r)/ \bar Y(N,\min(l_{\max},r),r) \le N
\end{equation}
Hence by changing $\gb$ by a factor $N$ again we can replace $V(r)$ by
 $\bar Y(N,\min(l_{\max},r),r)$.
The reader can check that the function $r\to \bar Y(N,\min(l_{\max},r),r)$ is log-concave and thus unimodal.
Hence it satisfies
  \begin{equation}\label{mixos4}
  \min (\sum_{q \le r} \bar Y(N,\min(l_{\max},q),q),  \sum_{q \ge r} \bar Y(N,\min(l_{\max},q),q))\le NY(N,\min(l_{\max},r),r).
 \end{equation}
and \eqref{mixos3} is satisfied with $\gb=C N^3$ (one multiplies by $N$ to have the inequality for $U$ and by $N$ again to have it for $\check Z^r$).
\end{proof}

\section{Proof of Theorem~\ref{th: mainresult2}}

In this section, we will make use of the techniques in \cite{BL13} and the estimates in the remainder of this article to establish Theorem~\ref{th: mainresult2}.

For this proof, we are going to make use the following result
\begin{theorem}[\cite{BL13}]
Let $X^N_t$ be Markovian processes on spaces $\Omega^N$ which are partitioned into $\mathcal{E}^N_1$ and $\mathcal{E}^N_2 = (\mathcal{E}_1^2)^c$. Then, supposing that
\begin{align}
  \label{e:pi_small}
  & \pi_N(\mathcal{E}^N_1) \ll \pi_N(\mathcal{E}^N_2), \text{ where $\pi^N$ is the stationary measure for $X^N_t$ and}\\
  \label{e:gap_large}
  & \Gap^N \ll \min\{ \Gap^N_1, \Gap^N_2 \}, \text{ where $\Gap^N_i$ is the spectral gap of $X^N_t$ restricted to $\mathcal{E}^N_i$.}
\end{align}
Then, starting from $\pi^N_1(\cdot) = \pi^N(\cdot|\mathcal{E}^N_1)$, the finite dimensional distributions of the process $\1_{\mathcal{E}^N_1}(X^N_{t \Tr})$ converge to that of $X_t$, where $X_t$ jumps from one to zero at rate one and then is absorbed.
\end{theorem}

\begin{proof}
The above theorem is not stated as above in \cite{BL13}, but we now indicate how to deduce such statements from this article. We have to verify conditions ({\bf L1}) and ({\bf L2G}) from Theorem~2.2 of \cite{BL13}. In view of \eqref{e:pi_small}, we can deduce ({\bf L2}) and consequently ({\bf L2G}).

Since we are dealing with only two valleys ($\mathcal{E}^N_1$ and $\mathcal{E}^N_2$), we can use Lemma~2.9 of \cite{BL13} to reduce ({\bf L1}) to \eqref{e:gap_large}. Below this lemma, there is an explanation of why the jump rates of the limiting process $X_t$ should be as stated, see also \eqref{e:pi_small} and the first paragraph in Subsection~F in \cite{BL13}.

Observe also that the time rescaling $\Tr$ corresponds to that in Lemma~2.9.
\end{proof}

\begin{proof}[Proof of Theorem~\ref{th: mainresult2}]
Joining the results in Theorem~\ref{t:pinned} and Propositions~\ref{p:activ_energy} and \ref{th:cornerflip}, we can easily obtain \eqref{e:pi_small} and \eqref{e:gap_large} whenever $\lambda > 2a/(1-2a)$ with $\lambda \neq \lambda_c(a)$. This implies Theorem~\ref{th: mainresult2}.
\end{proof}

\begin{remark}
It is a natural to ask whether one can obtain a stronger convergence than that stated in Theorem~\ref{th: mainresult2}. This would correspond for instance to establishing the condition ({\bf L4U}) in Lemma~2.5 of \cite{BL13} for properly chosen $\mathcal{E}_x$'s. It is clear that under our partition $\bar S^a_N$ and $\check S^a_N$ this cannot be true. By defining well separated sets $\bar{\mathcal{E}}$ and $\check{\mathcal{E}}_2$ corresponding to the free and pinned phases respectively, one could easily verify condition ({\bf L4U}) when $\eta \in \bar{\mathcal{E}}$. However the case $\eta \in \check{\mathcal{E}}$ seems more challenging and we leave it as an exciting open problem.

A less ambitious improvement that is possible to be obtained in Theorem~\ref{th: mainresult2} is the convergence of the semi-group (see Proposition~2.7 of \cite{BL13}). This requires us to prove ({\bf L4}), which we sketch below.
\end{remark}

\begin{remark}
For polymer dynamics there is another trail to follow in order to improve Theorem~\ref{th: mainresult2}, using monotonicity of the system.
As shown in \cite[Section 2]{CMT08}, pinning dynamics have nice order preserving properties.
In \cite[Theorem 1.3]{CLMST12}, this has been used to prove a metastable behavior for the polymer interacting with a repulsive interface which can be crossed which is slightly stronger than our result.
We believe that the proof of \cite{CLMST12} could in principle be replicated in our case, but we did not wish to reproduce a long proof here.
\end{remark}

Let us now estimate the probability of making the metastable transition from $\mathcal{R}$ to $\mathcal{R}^c$ before the relaxation within $\mathcal{R}$. This corresponds to the hypothesis {\bf (L4)} in \cite{BL13}. For this, we denote by $E^r$ and $P^r$ the expectation and probability measures governing the process $\eta_t$ reflected when exiting $\mathcal{R}$. Then, for any given time $T$,
\begin{equation}
  E^r_{\pi_{\mathcal{R}}} \Big[ \int_0^{T/2} \1_{X_s \in \partial \mathcal{R}} \d s\Big] = T \pi_{\mathcal{R}}(\partial \mathcal{R})/2,
\end{equation}
by Fubini. Moreover,
\begin{equation}
  \int_0^T \1_{X_s \in \partial \mathcal{R}} \dd s \geq \Big( \tfrac{T}{2} \wedge (H_{(\partial \mathcal{R})^c} \circ \theta_{H_{\partial \mathcal{R}}}) \Big)
  \cdot \1_{H_{\partial \mathcal{R}} < T/2},
\end{equation}
whose expectation is at least $c P_{\pi_\mathcal{R}} [H_{\partial \mathcal{R}} < T/2]$. Therefore
\begin{equation}
  P^r_{\pi_{\mathcal{R}}}[H_{\partial \mathcal{R}} < T] \leq c T \pi_{\mathcal{R}}.
\end{equation}
This proves ({\bf L4}) of \cite{BL13} with help of Proposition~\ref{upbound} and \eqref{eq:boundar}.

We should cite \cite{2011arXiv1103} at some point.

Let us now show that
\begin{equation}
  \limsup_N \frac{\tau_N \capacity_N(\bar S^a_N, \check S^a_N)}{\min\{\pi(\bar S^a_N), \pi(\check S^a_N)\}} < \infty.
\end{equation}
which corresponds to condition (2.10) of \cite{BL13}, see also (2.11).

\appendix

\section{Wall avoiding random walks}
\label{s:appendix}

\begin{lemma}\label{th:constraint}
For a fixed $a\in (0,1/2)$
for the symmetric nearest random walk on $\bbZ$
\begin{equation}
\lim_{N \to \infty} \; \; \sup_{l \in [\langle a N \rangle, N] \atop l \text{ even}} \Big| \bP\left[S_n>0, \ \forall n\in (0,l] \ | \ S_l=\langle aN \rangle \right] - \frac{\langle aN \rangle}{l} \Big| = 0.
\end{equation}
\end{lemma}
\begin{proof}
First let us consider a random-walk with drift $\frac{\langle aN \rangle}{l}$, i.e. with IID increments satisfying
$\tilde \bP[S_{n+1}-S_n=\pm 1]=\frac 1 2 \left(1\pm \frac{\langle aN \rangle}{l}\right)$
instead of the symmetric random-walk. As the walk is conditioned to $S_l=\langle aN \rangle$ one has
\begin{equation}
  \bP\left[S_n>0, \ \forall n\in (0,l] \ | \ S_l=\langle aN \rangle \right]=\tilde \bP\left[S_n>0, \ \forall n\in (0,l] \ | \ S_l=\langle aN \rangle \right].
\end{equation}
We now claim that the right hand side of the above equation does not change much if we drop the conditioning. More precisely
\begin{equation}
  \label{e:P_cond_P}
  \Big| \tilde \bP\left[S_n>0, \ \forall n\in (0,l] \ | \ S_l=\langle aN \rangle \right] - \tilde \bP\left[S_n>0, \ \forall n\in (0,l] \right] \Big| \leq c(a)N^{-1/10}.
\end{equation}
For this, we define a coupling $Q$ between the two above probabilities, which goes as follows.
\begin{itemize}
\item One first samples the random variables $\xi_1, \dots, \xi_l \in \{-1,1\}$ uniformly conditioned on their sum being $\langle aN \rangle$. Clearly, the partial sums $S_n$ of $\xi_i$'s have distribution $\tilde \bP [ \cdot | S_l = \langle aN \rangle]$.
\item Then, one samples a random variable $S_l$ under $\tilde \bP$ independently from the above and write $\Delta$ for the even number $S_l - \langle aN \rangle$.
\item Finally, if $\Delta \geq 0$ (respectively $\Delta < 0$) we flip $|\Delta|/2$ of the variables $\xi_i$ which had value $+1$ (respectively $-1$). We call $\xi_i'$ the modified increments and observe that their partial sum $S_n'$ has distribution $\tilde \bP [ \cdot ]$.
\end{itemize}

Now we can estimate
\begin{equation}
  \begin{split}
    \Big| \tilde \bP & \left[S_n>0, \ \forall n\in (0,l] \ | \ S_l=\langle aN \rangle \right] - \tilde \bP\left[S_n>0, \ \forall n\in (0,l] \right] \Big|\\
    & \leq Q[\xi_i \neq \xi'_i \text{ for some $i \leq \langle N^{1/6} \rangle$}] + Q[S_n \text{ touches $0$ after $\langle N^{1/6}\rangle$}]\\
    & \quad + Q[S_n' \text{ touches $0$ after $\langle N^{1/6}\rangle$}].
  \end{split}
\end{equation}
The last term above is clearly smaller or equal to $c(a) \exp\{-c'(a) N^{-1/6}\}$, by a large deviations bound. We now observe that $Q[S_n \text{ touches $0$ after $\langle N^{1/6}\rangle$}]$ is non-decreasing in $l$ (by a coupling argument), so we can assume that $l = N$. Again, a simple large deviations estimate is enough to bound this term by $c(a) \exp\{-c'(a) N^{-1/6}\}$.

To estimate $Q[\xi_i \neq \xi'_i \text{ for some $i \leq \langle N^{1/6} \rangle$}] $, we consider two separate cases:

\emph{Case 1} ($l - \langle aN \rangle \leq N^{2/3}$) - In this case, we expect both $S_n$ and $S_n'$ to give only upward steps before $N' = \langle N^{1/6} \rangle$. In fact a crude estimate gives
\begin{equation}
  E^Q[(N' - S_{N'})/2] \leq \frac{N'}{\langle aN \rangle}\frac{l - \langle aN \rangle}{2} \leq \frac{N^{1/6 + 2/3}}{2 \langle aN \rangle} \; \; \overset{N > c(a)}\leq \; \; \frac{N^{1/6 + 2/3 - 1}}{a} \leq c(a) N^{-1/6}.
\end{equation}
The same is true for $S'$ instead of $S$ by a very similar argument. This finishes the proof of \eqref{e:P_cond_P} for the first case.

\emph{Case 2} ($l - \langle aN \rangle \geq N^{2/3}$) - In this case, we expect that the third step in the construction of $Q$ (when we change some of the $\xi_i$ to $\xi_i'$) does not select any index $i \leq N'$ to be updated. This can be made precise by first estimating
\begin{equation}
  \var(\Delta) = l \frac{\big(1 + \tfrac{\langle aN \rangle}{l}\big)}{l} \frac{\big(1 - \tfrac{\langle aN \rangle}{2}\big)}{2} \leq l - \langle aN \rangle,
\end{equation}
so that
\begin{equation}
  P[|\Delta| \geq (l - \langle aN \rangle)^{3/5}] \leq (l - \langle aN \rangle)^{1 - 6/5} = (l - \langle aN \rangle)^{-1/5}.
\end{equation}

We can now evaluate
\begin{equation}
  \begin{split}
    & Q[\text{$S_i$ differs from $S'_i$ for some $i \leq N^{\langle N^{1/6} \rangle}$}] \leq E^Q\big[\#\{i \leq \langle N^{1/6} \rangle; \xi_i \neq \xi'_i\}\big]\\
    & \qquad \leq \frac{\langle N^{1/6} \rangle (l - \langle aN \rangle)^{3/5}}{\min \big\{(l-\langle aN \rangle)/2, (l+\langle aN \rangle)/2 \big\}} + P\big[|\Delta|\geq (l - \langle aN \rangle)^{3/5}\big]\\
    & \qquad \leq 2 N^{1/6} (l - \langle aN \rangle)^{-2/5} + N^{-1/10}\leq 3 N^{-1/10}.
  \end{split}
\end{equation}
This finishes the proof of \eqref{e:P_cond_P}.

We now conclude the proof of the Lemma, by observing that
\begin{equation}
\tilde\bP\left[ \exists n\ge (0,\sqrt{l}],\ S_n=0  \right] = \frac{\langle aN \rangle}{l}
\end{equation}
and combining this with \eqref{e:P_cond_P}.
\end{proof}

\section{Corner-flip dynamics with constraint}\label{cfcs}

 Let $M$ and $L$ be integers such that $|M|\le L$ and $L-M$ is even.
 Set
 $$\mathcal S_L^M:=\Big\{ \eta = (\eta_x)_{x\in [0,L]} \ | \ \eta_0=0,\ \eta_L=M, \text{ and } |\eta_{x+1}-\eta_{x}|=1 , \forall x\in [0,M]\Big\}.$$
 We define the partial order $\le$ on $\mathcal S_L^M$ by

\begin{equation}
\eta \le \eta' \Leftrightarrow \eta_x\le \eta'_x  \forall x\in [0,L].
\end{equation}

Now given, $\chi\leq \xi$ in $\mathcal S^L_M$, we define

\begin{equation}
 \mathcal S_L^M(\chi,\xi):=\Big\{ \eta\in \mathcal S^L_M \ | \ \chi\le \eta\le \xi \Big\}.
\end{equation}

The corner-flip dynamics on $\mathcal S^L_M(\chi,\xi)$ is defined by the following transition rates (recall the definition of $\eta^x$):
\begin{equation}
\begin{split}
 c(\eta,\eta^x) &= \tfrac 12 \ind_{\{ \eta^x\in  \mathcal S^L_M(\chi,\xi)\}},\\
  c(\eta,\eta')&=0 \quad  \text{ if } \eta'\notin \{ \eta^x \ | \ x\in \{1,\dots, L-1\}\}.
 \end{split}
\end{equation}

\begin{proposition}\label{th:cornerflip}
The spectral gap of the corner-flip dynamics in $\mathcal S^L_M(\chi,\xi)$ satisfies
\begin{equation}
\gap^M_L(\chi,\xi)\le (1-\cos(\pi/L)).
\end{equation}
\end{proposition}

\begin{proof}
The proof of this statement is done for the corresponding discrete time Markov chain in \cite{Wil04}, in the more general setup of
lozenge tiling. The statement appears in the last line of Table $1$ and is obtained by combining Theorem 7 with Lemma 1.
\end{proof}

\bibliographystyle{plain}
\bibliography{../BibTeX/all}

\def\cprime{$'$}
\begin{thebibliography}{10}

\bibitem{BL13}
Johel Beltr\'an and Claudio Landim.
\newblock A martingale approach to metastability.
\newblock Preprint available at http://arxiv.org/abs/1305.5987, 2013.

\bibitem{2011arXiv1103}
A.~{Bianchi} and A.~{Gaudilliere}.
\newblock {Metastable states, quasi-stationary and soft measures, mixing time
  asymptotics via variational principles}.
\newblock Preprint available at http://arxiv.org/abs/1103.1143, 2011.

\bibitem{BFO09}
Erwin Bolthausen, Tadahisa Funaki, and Tatsushi Otobe.
\newblock Concentration under scaling limits for weakly pinned {G}aussian
  random walks.
\newblock {\em Probab. Theory Related Fields}, 143(3-4):441--480, 2009.

\bibitem{CLMST12}
Pietro Caputo, Hubert Lacoin, Fabio Martinelli, Fran{\c{c}}ois Simenhaus, and
  Fabio~Lucio Toninelli.
\newblock Polymer dynamics in the depinned phase: metastability with
  logarithmic barriers.
\newblock {\em Probab. Theory Related Fields}, 153(3-4):587--641, 2012.

\bibitem{CMT08}
Pietro Caputo, Fabio Martinelli, and Fabio~Lucio Toninelli.
\newblock On the approach to equilibrium for a polymer with adsorption and
  repulsion.
\newblock {\em Electron. J. Probab.}, 13:no. 10, 213--258, 2008.

\bibitem{CDH11}
J.~{De Coninck}, F.~{Dunlop}, and T.~{Huillet}.
\newblock {Metastable wetting}.
\newblock {\em Journal of Statistical Mechanics: Theory and Experiment}, 6:13,
  June 2011.

\bibitem{MR2504175}
Frank den Hollander.
\newblock {\em Random polymers}, volume 1974 of {\em Lecture Notes in
  Mathematics}.
\newblock Springer-Verlag, Berlin, 2009.
\newblock Lectures from the 37th Probability Summer School held in Saint-Flour,
  2007.

\bibitem{Fischer84}
Michael~E. Fisher.
\newblock Walks, walls, wetting, and melting.
\newblock {\em J. Statist. Phys.}, 34(5-6):667--729, 1984.

\bibitem{FO10}
Tadahisa Funaki and Tatsushi Otobe.
\newblock Scaling limits for weakly pinned random walks with two large
  deviation minimizers.
\newblock {\em J. Math. Soc. Japan}, 62(3):1005--1041, 2010.

\bibitem{Gia07}
Giambattista Giacomin.
\newblock {\em Random polymer models}.
\newblock Imperial College Press, London, 2007.

\bibitem{Giaflour}
Giambattista Giacomin.
\newblock {\em Disorder and critical phenomena through basic probability
  models}, volume 2025 of {\em Lecture Notes in Mathematics}.
\newblock Springer, Heidelberg, 2011.
\newblock Lecture notes from the 40th Probability Summer School held in
  Saint-Flour, 2010, {\'E}cole d'{\'E}t{\'e} de Probabilit{\'e}s de
  Saint-Flour.

\bibitem{JSTV04}
Mark Jerrum, Jung-Bae Son, Prasad Tetali, and Eric Vigoda.
\newblock Elementary bounds on {P}oincar\'e and log-{S}obolev constants for
  decomposable {M}arkov chains.
\newblock {\em Ann. Appl. Probab.}, 14(4):1741--1765, 2004.

\bibitem{Lac13}
H.~{Lacoin}.
\newblock {The scaling limit of polymer pinning dynamics and a one dimensional
  Stefan freezing problem}.
\newblock Preprint available at http://arxiv.org/abs/1204.1253, 2012.

\bibitem{LPW09}
David~A. Levin, Yuval Peres, and Elizabeth~L. Wilmer.
\newblock {\em Markov chains and mixing times}.
\newblock American Mathematical Society, Providence, RI, 2009.
\newblock With a chapter by James G. Propp and David B. Wilson.

\bibitem{Wil04}
David~Bruce Wilson.
\newblock Mixing times of {L}ozenge tiling and card shuffling {M}arkov chains.
\newblock {\em Ann. Appl. Probab.}, 14(1):274--325, 2004.

\end{thebibliography}

\end{document}